\documentclass[a4paper,12pt]{article}
\usepackage[utf8]{inputenc}
\usepackage{amsthm}
\usepackage{amsfonts}
\usepackage{amssymb}
\usepackage{graphicx}
\usepackage{rotating}
\usepackage{wrapfig}
\usepackage[all]{xy}
\usepackage{pgf,tikz}
\usepackage{hyperref}
\usetikzlibrary{arrows}
\usetikzlibrary[patterns]

\newcommand{\N}{\mathbb{N}}

\newcommand{\U}{\mathfrak{U}}
\newcommand{\B}{\mathfrak{B}}

\newtheorem{prop}{Proposition}[section]
\newtheorem{cor}[prop]{Corollary}
\newtheorem{teo}[prop]{Theorem}

\newtheorem{lema}[prop]{Lemma}

\theoremstyle{definition}

\newtheorem*{obs}{Remark}

\newtheorem{defi}[prop]{Definition}

\newtheorem*{acknowledgements}{Acknowledgements}

\title{Equivariant blowups of bounded parabolic points}

\author{Lucas H. R. de Souza}

\begin{document}

\maketitle

\def\eod{\hfill$\square$}

\begin{abstract}Let $G$ be a group acting by homeomorphisms on a Hausdorff compact space $Z$. We constructed a new space $X$ that blows up equivariantly the bounded parabolic points of $Z$. This means, roughly speaking, that $G$ acts by homeomorphisms on $X$ and there exists a continuous equivariant map $\pi: X \rightarrow Z$ such that for every non bounded parabolic point $z \in Z$, $\#\pi^{-1}(z) = 1$.

We use such construction to characterize topologically some spaces that $G$ acts with the convergence property and to construct new convergence actions of $G$ from old ones. As one of the applications, if $G$ is a group and $p$ is a bounded parabolic point of the space of ends of $G$, then the stabilizer of $p$ is one-ended.
\end{abstract}

\let\thefootnote\relax
\let\thefootnote\relax

\tableofcontents

\section*{Introduction}

Let $G$ be a group with the discrete topology that acts properly discontinuously and cocompactly on a Hausdorff locally compact space $X$. A Hausdorff compact space $Y = X \cup C$, where the union is disjoint, $X$ is open and the action $G \curvearrowright X$ extends to an action by homeomorphisms on $Y$, has the perspectivity property if  for every $u$ entourage of the uniform structure compatible with $Y$ and every compact subset $K$ of $X$, the set $\{g \in G: g K \notin Small(u)\}$ is finite. In particular, we consider the case where $X = G$ with the left multiplication action. An example of it is Gerasimov's attractor-sum compactification \cite{Ge2} $G\cup C$ of $G$, if $G$ acts on $C$ with the convergence property (Proposition 7.5.4 of \cite{Ge2}). It generalizes Bowditch's compactification of relatively hyperbolic groups. This perspectivity property is also easily equivalent to one of the conditions of an $E\mathcal{Z}$-structure of a group (see \cite{Dr} for the definition).

Let $G$ be a group, $Z$ be a Hausdorff compact space, $\varphi: G \curvearrowright Z$ an action by homeomorphisms, $P \subseteq Z$ the subset of bounded parabolic points, $P' \subseteq P$ a subset of representatives of orbits of points in $P$ and for each $p\in P'$ a compact space $Y_{p} = Stab_{\varphi}p \cup C_{p}$ such that $Stab_{\varphi}p$ is open, the union is disjoint, the left multiplication action $G \curvearrowright G$ extends to an action by homeomorphisms on $Y_{p}$ and it has the perspectivity property. Then we construct a compact space $X$ and an action of $G$ on $X$ that blows up $Z$ through out the points of $P$ and the fibers are a copy of $C_{p}$ to each element of the orbit of the point $p \in P'$. Such construction is called parabolic blowup. It is characterized in the following sense:

\

$\!\!\!\!\!\!\!\!\!\!\!\!$ \textbf{Theorem \ref{surjectiveness}} Let $X$ be a Hausdorff compact space that $G$ acts by homeomorphisms, $\U$ the uniform structure compatible with the topology of $X$ and a continuous surjective equivariant map $\pi: X \rightarrow Z$ such that $\forall x \in Z-P$, $\#\pi^{-1}(x) = 1$. For $p \in P$, take $W_{p} = X - \pi^{-1}(p)$. Then $X$ is equivariantly homeomorphic to a parabolic blowup for some choice of $\{Y_{p}\}_{p\in P}$ if and only if the following conditions holds:

 \begin{enumerate}
    \item The map $\pi$ is topologically quasiconvex (i.e. $\forall u \in \U$ the set of inverse images of points of $Z$ that are not $u$-small is finite).

    \item $\forall p \in P$, $X$ has the perspectivity property with respect to $W_{p}$ and $Stab_{\varphi}p$ (i.e. $\forall u \in \U$, $\forall K$ compact subset of $W_{p}$, the set $\{g\in Stab_{\varphi} p: gK \notin Small(u) \}$ is finite.

\end{enumerate}

The map $\pi$ is called a blowup map, $Z$ is the base, $X$ the blowup space and $\{\pi^{-1}(p)\}_{p\in P}$ the fibers.

Dussaule, Gekhtman, Gerasimov and Potyagailo \cite{DGGP} described the Martin compactification of relatively hyperbolic groups with virtually abelian parabolic subgroups as a Dahmani boundary (in the sense of \cite{Da1}). To do this, they construct for those groups a boundary (called PBU-boundary, short for parabolic blow-up boundary), and proved that it is homeomorphic to both boundaries. An
essential part of \cite{DGGP} is devoted to prove the perspectivity property of the Martin compactification. 
Gerasimov and Potyagailo proposed to me a program of conversion of such construction of the PBU-boundary \cite{DGGP} to a general theory of parabolic blowups. The construction of the blowup and \textbf{Theorem \ref{surjectiveness}} of this preprint are the realization of this program. We decided to include some version of \textbf{Theorem \ref{surjectiveness}} and some other results of this preprint to a joint paper with Gerasimov and Potyagailo.

Similar constructions have already appeared in the literature. On the Combination Theorem \cite{Da2} Dahmani constructed the Bowditch boundary for fundamental groups of a class of graphs of groups such that the vertex groups are finitely generated and relatively hyperbolic. If all edges have finite groups, then his construction agrees with ours. Dahmani also constructed an $E\mathcal{Z}$-structure for relatively hyperbolic groups \cite{Da1} and Martin did a Combination Theorem for $E\mathcal{Z}$-structures of complexes of groups \cite{Ma}. $\acute{S}$wiatkowski worked on a special case of Martin's construction and characterized topologically a family of spaces that appears as boundaries of $E\mathcal{Z}$-structures of graphs of groups with finite edge groups \cite{Sw}. Such a family is called dense amalgams. We rewrote such characterization in a way that it is more compatible with our parabolic blowup construction (\textbf{Proposition \ref{denseamalgamequivalence}}). As an easy consequence, the Bowditch boundary of a fundamental group of a finite graph of groups such that the vertex groups are finitely generated and relatively hyperbolic and the edge groups are finite is a dense amalgam. Tshishiku and Walsh \cite{TW1} and \cite{TW2} constructed, from a relatively hyperbolic pair $(G,\mathcal{P})$ with Bowditch boundary homeomorphic to $S^{2}$ a Dahmani boundary (in the sense of the construction given in \cite{Da1}) homeomorphic to the Sierpi$\acute{n}$ski carpet. Our construction of the parabolic blowup as an inverse limit is analogous to their construction of the Sierpi$\acute{n}$ski carpet. 

There are two good properties of the blowup maps that worth mentioning:

\begin{enumerate}
    \item The parabolic blowup is functorial: Let $H$ be a group and $EPers(H)$ be the category whose objects are compact spaces, where $H$ is an open subset, together with an action by homeomorphisms from $H$ such that the action restricted to $H$ is the left multiplication action and the space is perspective with respect to $H$ and whose morphisms are continuous equivariant maps that are the identity when restricted to $H$. Let $Act(H)$ be the category whose objects are compact spaces together with an action of H by homeomorphisms and the morphisms are continuous equivariant maps. Then, fixed the map $\varphi: G \curvearrowright Z$ as described above, we have a functor $\prod_{p\in P'}EPers(Stab_{\varphi}p) \rightarrow Act(G)$ that send the family $\{Y_{p}\}_{p\in P}$ to the parabolic blowup with respect to that family (\textbf{Proposition \ref{fuctoriality}}).

    \item If $Z$ has a copy of $G$, then $X$ has also a copy of $G$ (\textbf{Proposition \ref{inclusaogrupo}}). Furthermore, if $Z$ has the perspectivity property with respect to $G$, then $X$ has the perspectivity property with respect to $G$ (\textbf{Proposition \ref{persppullback}}).
\end{enumerate}

If $G$ acts on two spaces $X$ and $Y$ with the convergence property and there is a continuous surjective equivariant map $\pi: X \rightarrow Y$ such that the inverse image of every non bounded parabolic point is one point, then $\pi$ is a blowup map (\textbf{Theorem \ref{semidirectforconv}}). If $G$ acts on two compact metrizable spaces $X$ and $Y$, where $Y$ comes from a geometric compactification of $G$, and $\pi: X \rightarrow Y$ is a continuous equivariant map, then the condition (1) of the \textbf{Theorem \ref{surjectiveness}} is not necessary to characterize blowup maps (\textbf{Theorem \ref{geoblowup}}).

If $G$ is a group acting on a CAT(0) space $X$ with isolated flats, Hruska and Kleiner \cite{HK} showed that $G$ is relatively hyperbolic with respect to the set $\mathcal{P}$ of stabilizers of maximal flats. Tran \cite{Tr} proved that there is a continuous equivariant map from the visual boundary $\partial_{v}(X)$ to the Bowditch boundary $\partial_{B}(G,\mathcal{P})$. Such map is an example of a blowup map (\textbf{Proposition \ref{catblowup}}) that does not come from a convergence action on the domain.

As applications of the parabolic blowup, we prove the following:

\begin{enumerate}
    \item Let $G$ be a group. If there is a parabolic blowup $\pi: X \rightarrow Z$ such that the attractor-sum (in the sense of \cite{Ge2}) $G\cup Z$ is geometric and the actions of the stabilizers of bounded parabolic points on the fibers have convergence property, then the action of $G$ on $X$ has the convergence property (\textbf{Theorem \ref{geoconv}}). In the special case where the actions on $Z$ and on all the fibers are relatively hyperbolic we also have that the action on $X$ is relatively hyperbolic (\textbf{Corollary \ref{relhiperb}}). We give a new proof of a Drutu and Sapir theorem \cite{DS} that says if $G$ is a finitely generated group that is relatively hyperbolic with respect to a set of parabolic subgroups $\mathcal{P}$ and each $P \in \mathcal{P}$ is relatively hyperbolic with respect to a set of parabolic subgroups $\mathcal{P}_{P}$, then $G$ is also relatively hyperbolic with respect to the set $\bigcup_{P \in \mathcal{P}}\mathcal{P}_{P}$. Furthermore, we construct its Bowditch boundary.
    \item If there is a parabolic blowup $\pi: X \rightarrow K$, where $X$ is metrizable and $K$ is the Cantor set, such that there are countably many bounded parabolic points on $K$ and each orbit of a bounded parabolic point of $K$ is dense, then $X$ is a dense amalgam (\textbf{Proposition \ref{caracterizacaotopologicapbucantor}}). This characterizes topologically metrizable spaces with more than two connected components that a finitely generated group acts with convergence property. On the opposite direction, given a finite graph of groups with  finitely generated vertex groups $\{G_{i}\}_{i\in V}$, finite edge groups and fundamental group $G$ and given a family of convergence actions of the vertex groups on spaces $C_{i}$, with $i \in V$, there is a convergence action of $G$ on a space $X$ such that $\forall i \in V$, $C_{i}$ is equivariantly homeomorphic to the limit set of $G_{i}$ on $X$ (\textbf{Proposition \ref{combinationconvergence}}). Such space $X$ is a dense amalgam of the spaces $C_{i}$. If the actions of the vertex groups are all relatively hyperbolic with $C_{i} = \partial_{B}(G_{i}, \mathcal{P}_{i})$, then the action on $X$ is also relatively hyperbolic and $X$ is the Bowditch boundary $\partial_{B}(G, \bigcup_{i}\mathcal{P}_{i})$ (we abuse the notation here saying that the vertex groups are subgroups of $G$).
    \item If $G$ is a group and $p$ is a bounded parabolic point of $Ends(G)$, then $Stab \ p$ has one end (\textbf{Proposition \ref{ends}}). We have two applications of the blowup construction that are also easy consequences of Bowditch's works about local cut points:  If $G$ is a finitely generated group that is relatively hyperbolic with respect to a set of subgroups $\mathcal{P}$, its Bowditch boundary $\partial_{B}(G,\mathcal{P})$ is connected and $G$ is not one-ended, then there exists $H \in \mathcal{P}$ that is not one-ended (\textbf{Corollary \ref{endsofsubgroups}}). And, if $G$ is a finitely generated group that acts minimally relatively hyperbolically on a non trivial space $X$ that is connected, locally connected and without local cut points, then $G$ is one-ended (\textbf{Corollary \ref{withoutlocalcut}}).
    \item If $G$ is a group and $p$ is a bounded parabolic point of the maximal perspective compactification of $G$, then the maximal perspective compactification of $Stab \ p$ is the one point compactification (\textbf{Theorem \ref{maximalpers}}).
\end{enumerate}

\begin{acknowledgements}This preprint contains a part of my Ph.D. thesis. The thesis was written under the advisorship of Victor Gerasimov, to whom I am grateful for our several discussions and lots of things that I've learned during this meantime.

I am also grateful to Leonid Potyagailo for his support and motivating discussions. I would like also to thank him and the Université de Lille for the invitation to a week-long travel, occasion that it was possible to have fruitful discussions (in special about part of the proof of \textbf{Theorem \ref{geoconv}}). 

I would like to thank Jason Manning for our useful discussions. In particular, he showed me how \textbf{Corollaries \ref{endsofsubgroups}} and  \textbf{\ref{withoutlocalcut}} comes easily from one of Bowditch`s theorems.

I would like to thank Christopher Hruska for showing me Dasgupta's thesis \cite{Das}.
\end{acknowledgements}

\section{Preliminaries}

\begin{prop}\label{uniaometrizavel}(Proposition 9, $\S4.4$, Chapter 1 of \cite{Bou}) Let $X$ be a Hausdorff compact space where there exists a family of subspaces $\{X_{n}\}_{n\in \N}$ such that each one has a countable basis and $X = \bigcup_{n\in \N}X_{n}$. Then, $X$ is metrizable.
\end{prop}

\begin{prop}\label{small}Let $f:(X_{1},\U_{1}) \rightarrow (X_{2},\U_{2})$ be an uniformly continuous map, $u \in \U_{2}$ and $Y \subseteq X_{2}: \ Y \in Small(u)$. Then $f^{-1}(Y) \in Small(f^{-1}(u))$. \eod
\end{prop}

\begin{prop}\label{liftingnet}Let $X$ and $Y$ be topological spaces and $f: X \rightarrow Y$ a continuous and closed map. If $\Phi = \{\{y_{\gamma,i}\}_{\gamma \in \Gamma}: i \in F\}$ is a set of nets in $Y$ that converges uniformly to a point $y\in Y$ such that $\#f^{-1}(y) = 1$, then  the set of liftings of $\Phi$ to $X$ (i.e. the nets $\{x_{\gamma}\}_{\gamma \in \Gamma}$ in $X$ such that $ \{f(x_{\gamma})\}_{\gamma \in \Gamma} \in \Phi$), converges uniformly to $f^{-1}(y)$.
\end{prop}

\begin{proof}Suppose that the set of liftings of $\Phi$ do not converge uniformly to $f^{-1}(y)$. So there exists an open neighbourhood $U$ of $f^{-1}(y)$ such that $\forall \gamma_{0}\in \Gamma$, $\exists \gamma_{1} > \gamma_{0}$, $\exists \{x_{\gamma,i}\}_{\gamma\in\Gamma}$ a lifting to the net $\{y_{\gamma,i}\}_{\gamma\in\Gamma}$ such that $x_{\gamma_{1},i} \notin U$. The set $X-U$ is closed and $f$ is a closed map, which implies that the set $f(X-U)$ is closed. Since $y \in V = Y-f(X-U)$, we have that $V$ is an open neighbourhood of $y$. So $y_{\gamma_{1},i} = f(x_{\gamma_{1},i}) \notin V$. We have that $\forall \gamma_{0}\in \Gamma$, $\exists \gamma_{1} > \gamma_{0}$, $\exists i \in F$ such that $y_{\gamma_{1},i} \notin V$. Thus $\Phi$ do not converge uniformly to $y$.
\end{proof}

\begin{defi}Let $X$ be a topological space, $\Phi = \{f_{\gamma}\}_{\gamma \in \Gamma}$  a net of maps from $X$ to itself and $A \subseteq X$. We denote by $\Phi|_{A} = \{f_{\gamma}|_{A}\}_{\gamma \in \Gamma}$. We say that $\Phi|_{A}$ converges uniformly to $x \in X$ if $\forall U$ neighbourhood of $x$, there exists $\gamma_{0} \in \Gamma$ such that $\forall \gamma > \gamma_{0}$, $f_{\gamma}(A) \subseteq U$.
\end{defi}

\begin{prop}\label{composicaouniforme}Let $X$ be a topological space, $\Phi = \{f_{\gamma}\}_{\gamma \in \Gamma}$ and $\Psi = \{g_{\gamma}\}_{\gamma \in \Gamma}$ be two families of maps from $X$ to itself and $A,B \subseteq X$. If $\Phi|_{A}$ converges uniformly to $a \in X$, $\Psi|_{B}$ converges uniformly to $b \in X$ and $B$ is a neighbourhood of $a$, then $(\Psi \circ \Phi)|_{A} = \{g_{\gamma} \circ f_{\gamma}|_{A}\}_{\gamma \in \Gamma}$ converges uniformly to $b$.
\end{prop}

\begin{proof}Let $U$ be a neighbourhood of $b$. Since $B$ is a neighbourhood of $a$, there exists $\gamma_{0} \in \Gamma$ such that $\forall \gamma > \gamma_{0}$, $f_{\gamma}(A) \subseteq B$. We have also that there exists $\gamma_{1} \in \Gamma$ such that $\forall \gamma > \gamma_{1}$, $g_{\gamma}(B) \subseteq U$. So if $\gamma_{2}$ is bigger than $\gamma_{0}$ and $\gamma_{1}$, then $\forall \gamma > \gamma_{2}$, $g_{\gamma}\circ f_{\gamma}(A) \subseteq g_{\gamma}(B) \subseteq U$. Thus, $(\Psi \circ \Phi)|_{A}$ converges uniformly to $b$.
\end{proof}

\subsection{Topological quasiconvexity}

\begin{defi}Let $X$ be a Hausdorff compact space and $\sim$ an equivalence relation on $X$. We say that $\sim$ is topologically quasiconvex if $\forall q \in X, \ [q]$ is closed and $\forall u \in \U$, $\#\{[x]\subseteq X: [x] \notin Small(u)\}< \aleph_{0}$, with $\U$ the only uniform structure compatible with the topology of $X$ and $[x]$ the equivalence class of $x$.

Let $X$, $Y$ be Hausdorff compact spaces. A quotient map $f: X \rightarrow Y$ is topologically quasiconvex if the relation $\sim = \Delta^{2}X \cup \bigcup_{y \in Y} f^{-1}(y)^{2}$ is topologically quasiconvex.
\end{defi}

\begin{prop}\label{quaseconvexidadetop}(Proposition 7.51 of \cite{So1}) Let $X$ be a Hausdorff compact space and $\sim$ a topologically quasiconvex equivalence relation on $X$. If $A \subseteq X/ \! \! \sim$, we define $\sim_{A} = \Delta^{2}X \cup\bigcup\limits_{[x] \in A} [x]^{2}$. Then, $\forall A \subseteq X/ \! \! \sim$, $X/ \! \! \sim_{A}$ is Hausdorff. \end{prop}

\begin{obs}In particular we have that $\sim_{X/\sim} = \sim$.
\end{obs}

\begin{prop}\label{transitivityquasiconvexitytopology}Let $X$ be a Hausdorff compact space, $\sim$ a topologically quasiconvex equivalence relation on $X$ and $\sim'$ an equivalence relation on $X$ such that $\sim' \subseteq \sim$ and $\forall q \in X$, $\sim' \! \cap [q]_{\sim}^{2}$ is topologically quasiconvex (as a relation in $[q]_{\sim}$), with $[q]_{\sim}$ the class of $q$ with respect to $\sim$. Then $\sim'$ is topologically quasiconvex.
\end{prop}

\begin{proof}Let $\U$ be the unique uniform structure compatible with the topology of $X$ and $u \in \U$. The set $\{[q]_{\sim}: [q]_{\sim} \notin Small(u)\}$ is finite. Let $[q]_{\sim} \notin Small(u)$ and consider $[j]_{\sim'}$ the equivalence class of a point $j$ with respect to $\sim'$. We have that $\{[j]_{\sim'}\subseteq [q]_{\sim}: [j]_{\sim'} \notin Small(u \cap [q]_{\sim}^{2})\}$ is finite, which implies that  $\{[j]_{\sim'}\subseteq [q]_{\sim}: [j]_{\sim'} \notin Small(u)\}$ is also finite. Then $\{[j]_{\sim'} \subseteq X: [j]_{\sim'} \notin Small(u)\}$ is finite, since it is a finite union of finite sets. Thus $\sim'$ is topologically quasiconvex.
\end{proof}

\begin{prop}\label{quaseconvexidadebase}Let $X$ be a Hausdorff compact space, $\U$ be the unique uniform structure compatible with the topology of $X$, $\B$ a base of $\U$ and $\sim$ an equivalence relation on $X$ such that $\forall u \in \B$, $\#\{[x]\subseteq X: [x] \notin Small(u)\}< \aleph_{0}$, with $[x]$ the class of $x$. Then $\sim$ is topologically quasiconvex.
\end{prop}

\begin{proof}Let $u \in \U$. We have that $\exists v\in \B:$ $v \subseteq u$. By hypothesis, we have that $\#\{[x]\subseteq X: [x] \notin Small(v)\}< \aleph_{0}$. If $[x] \in Small(v)$ and $v \subseteq u$ then $[x] \in Small(u)$, which implies that $\{[x]\subseteq X: [x] \notin Small(u)\} \subseteq \{[x]\subseteq X: [x] \notin Small(v)\}$. So $\#\{[x]\subseteq X: [x] \notin Small(u)\} < \aleph_{0}$. Thus $\sim$ is topologically quasiconvex.
\end{proof}

Consider $Z$ a Hausdorff compact space, $P \subseteq Z, \ \{X_{p}\}_{p \in P}$ a family of Hausdorff compact spaces, $\forall p \in P$, a continuous map $\pi_{p}: X_{p} \rightarrow Z$ that is injective at $X_{p} - \pi_{p}^{-1}(p)$ and surjective and $X = \lim\limits_{\longleftarrow}\{X_{p},\pi_{p}\}_{p\in P}$, together with the maps $\varpi_{p}: X \rightarrow X_{p}$, for $p \in P$.

We have that $\pi_{p}|_{X_{p} - \pi_{p}^{-1}(p)}: X_{p} - \pi_{p}^{-1}(p) \rightarrow Z - \{p\}$ is a homeomorphism. So there is a canonical copy of $Z-P$ on each $X_{p}$.

\begin{prop}\label{embedding}Let $\iota: Z - P \rightarrow X$ be the map inducted by the map inclusions of $Z-P$ on $X_{p}$, for each $p$, and on $Z$. Then, $\iota$ is an embedding.
\end{prop}

\begin{proof}The map $\iota$ is the unique continuous map that commutes the diagram (for each $p \in P$):

$$ \xymatrix{ Z - P   \ar[rd]_{\iota_{p}} \ar[r]^{\iota} & X \ar[d]^{\varpi_{p}} \ar[rd]^{\pi} & \\
            & X_{p} \ar[r]^{\pi_{p}} & Z} $$

Where $\iota_{p}$ is the inclusion map. Since $\pi_{p}\circ \iota_{p} = id_{Z-P}$ is injective, we have that $\iota$ is injective too. Let $A$ be a topological space and $g: A \rightarrow Z-P$ a map. If $g$ is continuous, then $\iota \circ g$ is continuous. If $\iota \circ g$ is continuous, then $\varpi_{p} \circ \iota \circ g$ is continuous (for every $p \in P$). But $\varpi_{p} \circ \iota \circ g = \iota_{p} \circ g$, which implies that $g$ is continuous, because $\iota_{p}$ is an embedding. So $g$ is continuous if and only if $\iota \circ g$ is continuous. Thus, $\iota$ is an embedding.
\end{proof}

\begin{prop}The projection map $\pi: X \rightarrow Z$ is injective on the set $\pi^{-1}(Z-P)$.
\end{prop}

\begin{proof}Consider the equivalence relation $\sim = \Delta^{2}(X) \cup \bigcup_{x\in  Z-P} \pi^{-1}(x)^{2}$. By the definition of this equivalence relation there exists maps $\varpi'_{p}$ and $\pi'$ that commutes the diagram (for each $p \in P$):

$$ \xymatrix{ & X \ar[ldd]_{\varpi_{p}} \ar[rdd]^{\pi} \ar[d]^{\varrho} & \\
            & X/ \! \! \sim \ar[ld]^{\varpi'_{p}} \ar[rd]_{\pi'} & \\
            X_{p}\ar[rr]^{\pi_{p}} & & Z} $$

Where $\varrho$ is the quotient map. Since $\varpi_{p}$ and $\pi$ are continuous and $\varrho$ is a quotient map, we have that $\varpi'_{p}$ and $\pi'$ are continuous. By the uniqueness of the limit, we have that $\varrho$ is a homeomorphism, which implies that $\sim = \Delta^{2}(X)$. Thus, $\forall x \in Z-P, \ \#\pi^{-1}(x) = 1$, which implies that $\pi|_{\pi^{-1}(Z-P)}$ is injective.

\end{proof}

\begin{prop}The map $\varpi_{p}|_{\pi^{-1}(p)}$ is a homeomorphism between $\pi^{-1}(p)$ and $\pi_{p}^{-1}(p)$.
\end{prop}

\begin{proof}Analogous to the proposition above.
\end{proof}

\begin{prop}\label{quaseconvexidadefibrado}Consider the relation $\sim = \Delta^{2} X^{2}\cup \bigcup\limits_{p\in P}\pi^{-1}(p)^{2}$ on the space $X$. Then $\sim$ is topologically quasiconvex.
\end{prop}

\begin{proof}Let $\U$ and $\U_{p}$  be the uniform structures compatible with the topologies of $X$ and $X_{p}$, respectively, and $\B = \{\varpi_{p_{1}}^{-1}(u_{p_{1}})\cap...\cap \varpi_{p_{n}}^{-1}(u_{p_{n}}): u_{p_{i}}\in \U_{p_{i}}\}$. The set $\B$ is a base for $\U$. Let $u = \varpi_{p_{1}}^{-1}(u_{p_{1}})\cap...\cap \varpi_{p_{n}}^{-1}(u_{p_{n}}) \in \B$ and $p\in P$ such that $\pi^{-1}(p) \notin Small(u)$. Then there exists $i \in \{1,...,n\}$ such that $\pi^{-1}(p)\notin Small(\varpi_{p_{i}}^{-1}(u_{p_{i}}))$,  which implies that $\varpi_{p_{i}}(\pi^{-1}(p)) \notin Small(u_{p_{i}})$ (\textbf{Proposition \ref{small}}). But $\varpi_{p_{i}}(\pi^{-1}(p))$ is a singleton for every $p_{i} \neq p$ and points are arbitrarily small. So $p \in \{p_{1},...,p_{n}\}$, which implies that the set $\{\pi^{-1}(p): \pi^{-1}(p)\notin Small(u)\}$ is finite. Since $u$ is an arbitrary element of $\B$, it follows that $\sim$ is topologically quasiconvex (\textbf{Proposition \ref{quaseconvexidadebase}}).
\end{proof}

\begin{prop}\label{invarianceofnonparabolics}Let's take a family of spaces $\{Y_{p}\}_{p\in P'}$, $\forall p \in P$, a continuous map $\pi'_{p}: Y_{p} \rightarrow Z$ that is injective at $Y_{p} - \pi'^{-1}_{p}(p)$ and surjective and $Y = \lim\limits_{\longleftarrow}\{Y_{p},\pi'_{p}\}_{p\in P}$. Take, for each $p \in P$, a continuous map $\phi_{p}: X_{p} \rightarrow Y_{p}$ such that the diagram commutes:

$$ \xymatrix{ X_{p}   \ar[rd]_{\pi_{p}} \ar[r]^{\phi_{p}} & Y_{p} \ar[d]^{\pi'_{p}} \\
            & Z} $$

Let $\phi: X \rightarrow Y$ be the induced map. Then the diagram commutes:

$$ \xymatrix{ Z - P   \ar[rd]_{\iota'} \ar[r]^{\iota} & X \ar[d]^{\phi} \\
            & Y} $$

Where $\iota$ and $\iota'$ are the embedding maps.
\end{prop}

\begin{proof}Consider the diagram:

$$ \xymatrix{ Z - P   \ar[rd]_{\iota'} \ar[r]^{\iota} & X \ar[d]^{\phi} \ar[r]^{\pi} & Z \ar[d]^{id}\\
            & Y \ar[r]^{\pi'} & Z} $$

Where $\pi$ and $\pi'$ are the projection maps. We have, by the definition of $\phi$, that the square commutes. Let $x \in Z-P$. So $\pi \circ \iota(x) = \pi' \circ \phi \circ \iota(x)$. But $\pi\circ \iota(x) = x = \pi' \circ \iota'(x)$, which implies that $\pi' \circ \iota'(x) = \pi' \circ \phi \circ \iota(x)$. Since $\pi'|_{\pi'^{-1}(Z-P)}$ is injective, we have that $\iota'(x) = \phi \circ \iota(x)$. Thus, the diagram commutes.
\end{proof}

Now we show that topologically quasiconvex relations and inverse limits as described above are related.

\begin{prop}Let $W$ be a Hausdorff compact space, $\sim$ an equivalence relation on $W$, $Z = W/ \! \! \sim$ and $X_{p} = W / \! \! \sim_{p}$, where $p\in Z$ and $\sim_{p} = \sim_{W-\{p\}}$. Let $\pi_{p}: X_{p} \rightarrow Z$ be the quotient map and $X = \lim\limits_{\longleftarrow}\{X_{p},\pi_{p}\}_{p\in Z}$. If $\forall p \in Z, X_{p}$ is Hausdorff, then $X \cong W$.
\end{prop}

\begin{proof}Let $\pi: W \rightarrow Z$ and $\forall p \in Z$, $\varpi_{p}: W \rightarrow X_{p}$ be the quotient maps and $\pi': X \rightarrow Z$ and  $\varpi_{p}': X \rightarrow X_{p}$ be the projection maps. We have that $\pi$ and the maps $\varpi_{p}$ induct a continuous map $f: W \rightarrow X$. Let $w,w' \in W$ such that $f(w) = f(w')$. We have that $\varpi_{\pi(w)}|_{\pi^{-1}(\pi(w))}$ and $\varpi'_{\pi'(f(w))}|_{\pi'^{-1}(\pi'(f(w)))}$ are homeomorphisms under their images, which implies that $f|_{\pi^{-1}(\pi(w))} = (\varpi'_{\pi'(f(w))}|_{\pi'^{-1}(\pi'(f(w)))})^{-1} \circ  \varpi_{\pi(w)}|_{\pi^{-1}(\pi(w))}$ is a bijection. Since $w,w' \in \pi^{-1}(\pi(w))$, it follows that $w = w'$. So $f$ is injective. Let $x \in X$. We have that  $f|_{\pi^{-1}(\pi'(x))}: \pi^{-1}(\pi'(x)) \rightarrow \pi'^{-1}(\pi'(x))$ is a bijection. Then there exists $w \in \pi^{-1}(\pi'(x))$ such that $f(w) = x$. Thus $f$ is a bijection and then a homeomorphism, since $W$ is compact and $X$ is Hausdorff.
\end{proof}

\begin{cor}\label{equivalencetopconvexity}Let $W$ be a Hausdorff compact space, $\sim$ an equivalence relation on $W$, $Z = W/ \! \! \sim$ and $X_{p} = W / \! \! \sim_{p}$, where $p\in Z$ and $\sim_{p} = \sim_{W-\{p\}}$. Then $\sim$ is topologically quasiconvex if and only if $\forall p \in Z$, $X_{p}$ is Hausdorff. \eod
\end{cor}

\subsection{Dense amalgam}

In \cite{Sw} $\acute{S}$wiatkowski proved the following:

\begin{prop}Let $C_{1},...,C_{n}$ be compact metrizable spaces. Then there is a unique, up to homeomorphisms, metrizable space $X$ and $\mathcal{Y} = \mathcal{Y}_{1}\dot{\cup}...\dot{\cup}\mathcal{Y}_{n}$ a family of pairwise disjoint subsets of $X$ satisfying the following:

\begin{enumerate}
    \item $\forall i \in \{1,...,n\}$, $\mathcal{Y}_{i}$ is infinite countable and $\forall C \in \mathcal{Y}_{i}$, $C$ is homeomorphic to $C_{i}$.
    \item If $\U$ is the uniform structure compatible with the topology of $X$ and $u \in \U$, then the set $\{C\in \mathcal{Y}: C \notin Small(u)\}$ is finite.
    \item $\forall C \in \mathcal{Y}$, $X-C$ is dense in $X$.
    \item $\forall i \in \{1,...,n\}$, $\bigcup\mathcal{Y}_{i}$ is dense in $X$.
    \item Let $x,y \in X$ such that there is no $C \in \mathcal{Y}$ such that  $x,y \in C$. Then, there is a clopen set $F \subseteq X$ such that $x \in F$, $y \notin F$ and $\forall C \in \mathcal{Y}$, $C\subseteq F$ or $C\cap F = \emptyset$.
\end{enumerate}

\end{prop}

He calls such space the dense amalgam of the spaces $C_{1},...,C_{n}$.

On this subsection, consider $K$ the Cantor set. Let $X$ be the dense amalgam of $C_{1},...,C_{n}$ and consider the equivalence relation $\sim = \Delta X \cup \bigcup_{C \in \mathcal{Y}} C^{2}$. Let $\pi: X \rightarrow X/\sim$ be the quotient map. From the second condition we have that $\pi$ is topologically quasiconvex. From Lemma 2.A.1 of \cite{Sw}, $X/ \! \! \sim$ is homeomorphic to $K$.

\begin{lema}$\forall i \in \{1,...,n\}$, $\pi(\bigcup\mathcal{Y}_{i})$ is dense on $X/ \! \! \sim$.
\end{lema}

\begin{proof}Since $\pi$ is continuous and surjective, $X/ \! \! \sim = \pi(Cl_{X}(\bigcup\mathcal{Y}_{i}))  \subseteq \\ Cl_{X/\sim}(\pi(\bigcup\mathcal{Y}_{i}))$, which implies that $\pi(\bigcup\mathcal{Y}_{i})$ is dense on $X/ \! \! \sim$ .
\end{proof}

Let now $X$ be a metrizable compact space and $\pi: X \rightarrow K$ a topologically quasiconvex map such that there is $P = P_{1}\dot{\cup}...\dot{\cup} P_{n}$, a subset of $K$, satisfying:

\begin{enumerate}
    \item $\forall i \in \{1,...,n\}$, $P_{i}$ is countable and dense on $K$.
    \item If $p,q \in P_{i}$, then $\pi^{-1}(p)$ and $\pi^{-1}(q)$ are homeomorphic.
    \item If $p \notin P$, then $\#\pi^{-1}(p) = 1$.
    \item $\forall p \in P$, $X-\pi^{-1}(p)$ is dense on $X$.
\end{enumerate}

\begin{lema}$X - \pi^{-1}(P)$ is dense on $X$.
\end{lema}

\begin{proof}Let $p \in P$. We have that $X- \pi^{-1}(p)$ is open and dense on $X$. Since $X$ is a Baire space and $P$ is countable, we have that $X - \pi^{-1}(P) = \bigcap_{p\in P} X- \pi^{-1}(p)$ is dense on $X$.
\end{proof}

\begin{lema}$\forall i \in\{1,..,n\}$, $\pi^{-1}(P_{i})$ is dense on $X$.
\end{lema}

\begin{proof}Let $U$ be an open set of $X$. Since $X-\pi^{-1}(P)$ is dense on $X$, there exists $x \in U\cap (X-\pi^{-1}(P))$. Since $P_{i}$ is dense on $K$, there exists a sequence $\{y_{n}\}_{n\in \N} \subseteq P_{i}$ that converges to $\pi(x)$. Since $\pi^{-1}(x) = \{x\}$, we have, by \textbf{Proposition \ref{liftingnet}}, that if $\{x_{n}\}_{n \in \N}$ is a sequence on $X$ such that $\forall n \in \N$, $x_{n} \in \pi^{-1}(y_{n})$, then $\{x_{n}\}_{n \in \N}$ converges to $x$. But $\{x_{n}\}_{n \in \N}$ is a sequence contained in $\pi^{-1}(P_{i})$. Thus $\pi^{-1}(P_{i})$ is dense on $X$.
\end{proof}

Summarising we get:

\begin{prop}\label{denseamalgamequivalence}Let $X$ be a metrizable compact space. The space $X$ is a dense amalgam if and only if there is a map $\pi: X \rightarrow K$ that is topologically quasiconvex and $P = P_{1}\dot{\cup}...\dot{\cup}P_{n}$ a countable subset of $K$ such that $\forall i \in \{1,...,n\}$, $P_{i}$ is dense on $K$, $\forall p \in K-P$, $\#\pi^{-1}(p) = 1$, $\forall p \in P$, $X-\pi^{-1}(p)$ is dense on $X$ and $\forall p,q \in P_{i}$, $\pi^{-1}(p)$ and $\pi^{-1}(q)$ are homeomorphic. Moreover, the family $\mathcal{Y}$ is given by $\mathcal{Y}_{i} = \{\pi^{-1}(p): p \in P_{i}\}$. \eod
\end{prop}

\subsection{Group actions}

\begin{defi}Let $\varphi: G \curvearrowright X$ be an action by homeomorphisms. A point $p \in X$ is bounded parabolic if the action $\varphi|_{Stab_{\varphi}p\times X-\{p\}}: Stab_{\varphi}p \curvearrowright X-\{p\}$ is properly discontinuous and cocompact.
\end{defi}

\begin{prop}(Lemma 1.9 of \cite{Bo2}) Let $G$ be a group that acts by homeomorphisms on two spaces $Y$ and $Z$ and $f: Y \rightarrow Z$ a proper continuous surjective $G$-equivariant map. Then, $G$ acts properly discontinuously on $Y$ if and only if $G$ acts properly discontinuously on $Z$. Also, $G$ acts cocompactly on $Y$ if and only if $G$ acts cocompactly on $Z$.
\end{prop}

\begin{cor}\label{liftingparabolic}Let $\alpha: G \curvearrowright Y$ and $\varphi: G \curvearrowright Z$ be actions by homeomorphisms and $f: Y \rightarrow Z$ a proper continuous and $G$-equivariant map. If $p \in Z$ is a bounded parabolic point, then the action of $\alpha|_{Stab_{\varphi} p\times (Y - f^{-1}(p))}$ is properly discontinuous and cocompact. \eod
\end{cor}

\begin{cor}\label{liftingparabolic2}Let $G$ be a group that acts by homeomorphisms on two spaces $Y$ and $Z$ and $f: Y \rightarrow Z$ a proper continuous surjective $G$-equivariant map. Let $p \in Z$ such that $\#f^{-1}(p) = 1$. Then, $p$ is bounded parabolic if and only if $f^{-1}(p)$ is bounded parabolic. \eod
\end{cor}

\begin{defi}Let $G$ be a group, $X$ a Hausdorff compact space and $\varphi: G \curvearrowright X$ an action by homeomorphisms. We say that $\varphi$ has the convergence property if for every wandering net $\Phi$ (i.e. a net such that two elements are always distinct) has a subnet $\Phi'$ such that there exists $a,b \in X$ such that $\Phi'|_{X-\{b\}}$ converges uniformly to $a$.
We say that $\Phi'$ is a collapsing net with attracting point $a$ and repelling point $b$.

\end{defi}

\begin{defi}Let $G$ be a group acting on a Hausdorff compact space $X$ with convergence property. The attractor-sum compactification of $G$ is the unique compactification of $G$ with $X$ is its remainder such that the action of $G$ on it (that is the left multiplication on $G$ and the previous action on $X$) has the convergence property.
\end{defi}

\begin{obs}The existence and uniqueness of the attractor-sum compactification is due to Gerasimov (Proposition 8.3.1 of \cite{Ge2}).

\end{obs}

\begin{defi}Let $\varphi: G \curvearrowright X$ be a convergence action and $p \in X$. We say that $p$ is a conical point if there is an infinite set $K \subseteq G$ such that $\forall q \neq p$, $Cl_{X}(\{(\varphi(g,p),\varphi(g,q)): g \in K\})\cap \Delta X = \emptyset$.
\end{defi}

\begin{defi}Let $G$ be a finitely generated group, $X$ a Hausdorff compact space and $\varphi: G \curvearrowright X$ an action by homeomorphisms. We say that $\varphi$ is hyperbolic if it has the convergence property and the induced action on the space of distinct triples of $X$ is cocompact.
\end{defi}

\begin{prop}(Bowditch, Lemma 1.4 of \cite{Bo2}) If $X$ is metrizable and $\varphi$ is a minimal convergence action, then $\varphi$ is hyperbolic if and only if every point of $X$ is conical.
\end{prop}

\begin{defi}Let $G$ be a finitely generated group, $X$ a Hausdorff compact space and $\varphi: G \curvearrowright X$ an action by homeomorphisms. We say that $\varphi$ is relatively hyperbolic if it has the convergence property, its limit set is bigger than one point and the induced action on the space of distinct pairs of $X$ is cocompact. If $\mathcal{P}$ is a representative set of conjugation classes of stabilizers of bounded parabolic points of $X$, then we say that $\varphi$ is relatively hyperbolic with respect to $\mathcal{P}$.
\end{defi}

\begin{prop}(Tukia, 1C of \cite{Tu} - Gerasimov, Main Theorem of \cite{Ge1}) If $X$ is metrizable and the action $\varphi$ is a minimal convergence action, then $\varphi$ is relatively hyperbolic if and only if every point of $X$ is conical or bounded parabolic.
\end{prop}

\begin{defi}Let $G$ be a group, $S$ a finite set of generators of $G$ and a compact metrizable space $X = G\dot{\cup}Y$ where $G$ is discrete and open. An action by homeomorphisms $\varphi: G \curvearrowright X$ is geometric with respect to $S$ if $\varphi|_{G\times G}$ is the left multiplication action and $\forall u \in \U$, $\exists K \subseteq G$ finite such that every geodesic $\gamma$ (segment, ray or line) on $G$ relative to $S$ (i.e. $\gamma$ is the set of points in $G$ of a geodesic in $Cay(G,S)$) that do not intersect $K$ is $u$-small, where $\U$ is the unique uniform structure compatible with the topology of $X$.
\end{defi}

We have that relatively hyperbolic actions of finitely generated groups are geometric \cite{Ge2} and geometric actions have the convergence property (Proposition 4.3.1 of \cite{Ge2}).

\subsection{Artin-Wraith glueings}

\begin{defi}Let $X$ and $Y$ be topological spaces and an application $f: Closed(X) \rightarrow Closed (Y)$ such that $\forall A,B \in Closed(X), \ f(A\cup B) = f(A)\cup f(B)$ and $f(\emptyset) = \emptyset$ (we will say that such map is admissible). We will give a topology for $X\dot{\cup} Y$. Let's declare as a closed set $A \subseteq X\dot{\cup} Y$ if $A \cap X \in Closed(X), \ A \cap Y \in Closed(Y)$ and $f(A\cap X) \subseteq A$. Therefore, let's denote by $\tau_{f}$ the set of the complements of this closed sets and $X+_{f}Y = (X\dot{\cup} Y,\tau_{f})$.
\end{defi}

\begin{defi}Let $X+_{f} \! Y$ and $Z+_{h} \! W$ be topological spaces and continuous maps $\psi: X \rightarrow Z$ and $\phi: Y \rightarrow W$. We define $\psi + \phi: X+_{f}Y \rightarrow Z+_{h}W$ by $(\psi + \phi)(x) = \psi(x)$ if $x \in X$ and $\phi(x)$ if $x \in Y$. If $G$ is a group, $\psi: G \curvearrowright X$ and $\phi: G \curvearrowright Y$, then we define $\psi+\phi: G \curvearrowright X+_{f}Y$ by $(\psi+\phi)(g,x) = \psi(g,x)$ if $x \in X$ and $\phi(g,x)$ if $x \in Y$.
\end{defi}

\begin{prop}(Proposition 4.2 of \cite{So1}) Let $X+_{f}Y$ and $Z+_{h}W$ be topological spaces and $\psi: X \rightarrow Z$ and $\phi: Y \rightarrow W$ continuous maps. Then, $\psi + \phi: X+_{f}Y \rightarrow Z+_{h}W$ is continuous if and only if $\forall A \in Closed(Z), \ f(\psi^{-1}(A)) \subseteq \phi^{-1}(h(A))$. In another words, we have the diagram:

$$ \xymatrix{ Closed(X) \ar[r]^{f} & Closed(Y) \\
            Closed(Z) \ar[r]^{h} \ar[u]^{\psi^{-1}} \ar@{}[ur]|{\subseteq} & Closed(W) \ar[u]^{\phi^{-1}} } $$

\end{prop}

\begin{defi}Let $X+_{f}W$, $Y$ and $Z$ be topological spaces and $\Pi: Closed(Y) \rightarrow Closed (X)$ and $\Sigma: Closed(W) \rightarrow Closed(Z)$  admissible maps. We define the admissible map $f_{\Sigma\Pi}: Closed(Y) \rightarrow Closed(Z)$ as $f_{\Sigma\Pi} = \Sigma \circ f \circ \Pi$.
\end{defi}

\begin{prop}(Cube Lemma - Proposition 5.3 of \cite{So1}) Let $X_{i}+_{f_{i}}W_{i}$, $Y_{i}$ and $Z_{i}$ be topological spaces, $\Pi_{i}: Closed(Y_{i}) \rightarrow Closed(X_{i})$ and $\Sigma_{i}: Closed(W_{i}) \rightarrow Closed(Z_{i})$ admissible maps. Take the respective induced maps $f_{i\Sigma_{i}\Pi_{i}}: Closed(Y_{i}) \rightarrow Closed(Z_{i})$. If $\mu+\nu: X_{1}+_{f_{1}}W_{1} \rightarrow X_{2}+_{f_{2}}W_{2}$, $\psi: Y_{1} \rightarrow Y_{2}$ and $\phi: Z_{1} \rightarrow Z_{2}$ are continuous maps that form the diagrams:

$$ \xymatrix{   Closed(X_{2}) \ar[r]^{\mu^{-1}} \ar@{}[dr]|{\supseteq} & Closed(X_{1}) & & Closed(W_{2}) \ar[r]^{\nu^{-1}} \ar[d]^{\Sigma_{2}} \ar@{}[dr]|{\supseteq} & Closed(W_{2}) \ar[d]_{\Sigma_{1}} \\
                Closed(Y_{2}) \ar[r]^{\psi^{-1}} \ar[u]^{\Pi_{2}} & Closed(Y_{1}) \ar[u]_{\Pi_{1}} & & Closed(Z_{2}) \ar[r]^{\phi^{-1}} & Closed(Z_{1}) } $$

Then, $\psi+\phi: Y_{1}+_{f_{1\Sigma_{1}\Pi_{1}}}Z_{1}\rightarrow Y_{2}+_{f_{2\Sigma_{2}\Pi_{2}}}Z_{2}$ is continuous.
\end{prop}

\begin{defi}Let $X+_{f}W$, $Y$ and $Z$ be topological spaces, $\pi: Y \rightarrow X$ and $\varpi: Z \rightarrow W$ be two continuous maps. We define the pullback of $f$ with respect to $\pi$ and $\varpi$ by $f^{\ast}(A) = \varpi^{-1}(f(Cl_{X}\pi(A)))$.
\end{defi}

\begin{prop}\label{pullbackaction}(Corollary 2.5 of \cite{So}) Let $G_{1},G_{2}$ be groups, $X+_{f}W, Y$ and $Z$ topological spaces, $\alpha: G_{1}\rightarrow G_{2}$ a homomorphism, $\pi: Y \rightarrow X$ and $\varpi: Z \rightarrow W$ continuous maps and $f^{\ast}: Closed(Y) \rightarrow Closed(Z)$ the pullback of $f$. Let $\mu+\nu: G_{2} \curvearrowright X+_{f}W$, $\psi: G_{1} \curvearrowright Y$ and $\phi: G_{1} \curvearrowright Z$ be actions by homeomorphisms such that $\pi$ and $\varpi$ are $\alpha$ - equivariant. Then, the action $\psi+\phi: G_{1} \curvearrowright Y+_{f^{\ast}}Z$ is by homeomorphisms.
\end{prop}

\subsection{Attractor-sum functors}

\begin{defi}\label{pielambda}Let $G$ be a group, $X$ and $Y$ Hausdorff topological spaces with $X$ locally compact and $Y$ compact, $L: G \curvearrowright G$ the left multiplication action, $\varphi: G \curvearrowright X$ a  properly discontinuous cocompact action, $\psi: G \curvearrowright Y$ an action by homeomorphisms and $K \subseteq X$ a compact such that $\varphi(G,K) = X$. Define $\Pi_{K}: Closed(X) \rightarrow  Closed(G)$ as $\Pi_{K}(S) = \{g\in G: \varphi(g,K)\cap S \neq \emptyset\}$ and $\Lambda_{K}: Closed(G) \rightarrow Closed(X)$ as $\Lambda_{K}(F) = \varphi(F,K)$.
\end{defi}

\begin{defi}Let $G$ be a group, $X,Y$ Hausdorff spaces with $X$ locally compact and $Y$ compact, $\psi: G \curvearrowright Y$ by homeomorphisms and $\varphi: G \curvearrowright X$ properly discontinuous and cocompact. We say that a compact space of the form $X+_{f}Y$ is perspective if it is Hausdorff, $\varphi+\phi: G \curvearrowright X+_{f}Y$ is continuous and $\forall u \in \mathcal{U}_{f}, \ \forall K \subseteq X$ compact, $\#\{g \in G: \varphi(g,K)\notin Small(u)\} < \aleph_{0}$, where $\mathcal{U}_{f}$ is the only uniform structure of $X+_{f}Y$.

We denote by $T_{2}EComp(\varphi)$ the category whose objects are Hausdorff compact spaces of the form $X+_{f}Y$, for some Hausdorff compact $Y$, with an action by homeomorphisms of the form $\varphi+\psi: G \curvearrowright X+_{f}Y$, and morphisms are continuous maps of the form $id+\phi: X+_{f_{1}}Y_{1}\rightarrow X+_{f_{2}}Y_{2}$, such that $\phi$ is equivariant with respect to the actions of $G$ on $Y_{1}$ and $Y_{2}$. We denote by $EPers(\varphi)$ the full subcategory of $EComp(\varphi)$ whose objects are spaces with the perspective property and $EMPers(\varphi)$ the full subcategory of $EPers(\varphi)$ whose objects are metrizable spaces.
\end{defi}

\begin{prop}\label{functor}(Proposition 3.10 of \cite{So}) Let $G_{1}$, $G_{2}$ be groups, $X_{1}$, $X_{2}$ Hausdorff locally compact spaces, $Y_{1}$, $Y_{2}$ Hausdorff compact spaces, $\varphi_{i}: G_{i} \curvearrowright X_{i}$ properly discontinuous cocompact actions, $\psi_{i}: G_{i} \curvearrowright Y_{i}$ actions by homeomorphisms, $\alpha: G_{1} \rightarrow G_{2}$ a homomorphism, $\mu: X_{1} \rightarrow X_{2}$, $\nu: Y_{1} \rightarrow Y_{2}$ continuous maps with $\mu$ $\alpha$-equivariant, $K_{i} \subseteq X_{i}$ compact subspaces such that $\varphi_{i}(G,K_{i}) = X_{i}$ and $\mu(K_{1}) \subseteq K_{2}$ and $\Pi_{K_{i}}: Closed(X_{i}) \rightarrow Closed(G_{i})$. If the application $\alpha+\nu: G_{1}+_{\partial_{1}}Y_{1} \rightarrow G_{2}+_{\partial_{2}}Y_{2}$ is continuous, then the map $\mu+\nu: X_{1}+_{\partial_{1\Pi_{K_{1}}}} Y_{1} \rightarrow X_{2}+_{\partial_{2\Pi_{K_{2}}}} Y_{2}$ is continuous.
\end{prop}

\begin{prop}(Theorem 3.2 of \cite{So}) Let $K \subseteq X$ be a fundamental domain of $\varphi$. The functor $\Pi: EPers(G) \rightarrow EPers(\varphi)$ that sends $G+_{\partial}Y$ to  $X+_{\partial_{\Pi_{K}}}Y, \ L+\psi:G \curvearrowright G+_{\partial}Y$ to $\varphi+\psi: G \curvearrowright X+_{\partial_{\Pi_{K}}}Y$ and $id+\phi: G+_{\partial_{1}}Y_{1}\rightarrow G+_{\partial_{2}}Y_{2}$ to $id+\phi: X+_{(\partial_{1})_{ \Pi_{K}}} Y_{1}\rightarrow X+_{(\partial_{2})_{\Pi_{K}}}Y_{2}$, is a isomorphism of categories.

Furthermore, its inverse is the functor $\Lambda: EPers(\varphi) \rightarrow EPers(G)$ that sends $X+_{f}Y$ to $G+_{f_{\Lambda_{K}}}Y, \ \varphi+\psi:G \curvearrowright X+_{f}Y$ to $id+\psi: G \curvearrowright G+_{f_{\Lambda_{K}}}Y$ and $id+\phi: X+_{f_{1}}Y_{1}\rightarrow X+_{f_{2}}Y_{2}$ to $id+\phi: G+_{(f_{1})_{\Lambda_{K}}}Y_{1}\rightarrow G+_{(f_{2})_{\Lambda_{K}}}Y_{2}$.
\end{prop}

Since $\Pi$ and $\Lambda$ do not depend of the choice of the fundamental domain (Propositions 3.18 and 3.20 of \cite{So}), we denote $\partial_{\Pi_{K}}$ by $\partial_{\Pi}$ and $f_{\Lambda_{K}}$ by $f_{\Lambda}$.

\begin{cor}(Corollary 3.5 of \cite{So}) Let $G$ be a countable group, $X$ a locally compact Hausdorff space with countable basis and $\varphi: G \curvearrowright X$ properly discontinuous. Then, the functor $\Pi$ maps $EMPers(G)$ to $EMPers(\varphi)$ isomorphically.
\end{cor}

\begin{prop}\label{quocienteperspectivo}(Proposition 4.10 of \cite{So}) Let $\varphi_{i}+\psi_{i}: G \curvearrowright X_{i}+_{f_{i}}Y_{i}$, for $i =1,2$, be actions by homeomorphisms on Hausdorff spaces and let $m+n: X_{1}+_{f_{1}}Y_{1} \rightarrow X_{2}+_{f_{2}}Y_{2}$ be a continuous and surjective map with $m$ equivariant. If $X_{1}+_{f_{1}}Y_{1}$ is perspective, then $X_{2}+_{f_{2}}Y_{2}$ is perspective.
\end{prop}

\section{Parabolic blowup}

\subsection{Construction}

Let $G$ be a group, $Z$ a compact Hausdorff space, $\varphi: G \curvearrowright Z$ an action by homeomorphisms, $P \subseteq Z$ the set of bounded parabolic points of $\varphi$, $P' \subseteq P$ a representative set of orbits, $\mathcal{C} = \{C_{p}\}_{p\in P'}$ a family of compact Hausdorff spaces, $H = \{H_{p}\}_{p\in P'}$, with $H_{p} \subseteq G$ minimal sets such that $1\in H_{p}$ and $Orb_{\varphi|_{H_{p}\times Z}} p = Orb_{\varphi} p$ and $\{Stab_{\varphi} p+_{\partial_{p}}C_{p}\}_{p\in P'}$ a family of spaces with the equivariant perspective property with actions $\eta = \{\eta_{p}\}_{p\in P'}$, such that the action $L_{p}+\eta_{p}: Stab_{\varphi} p \curvearrowright Stab_{\varphi} p +_{\partial_{p}}C_{p}$  is by homeomorphisms (where $L_{p}$ is the left multiplication action).

We define for $p \in P'$ and $q \in Orb_{\varphi} p, \ C_{q} = C_{p}$ and define $h_{p,q}$, for $q \in Orb_{\varphi} p$, the only element of $H_{p}$ such that $\varphi(h_{p,q},p) = q$. Let $p \in P'$ and $q \in Orb_{\varphi} p$. We have that $Stab_{\varphi} q = h_{p,q} (Stab_{\varphi} p) h_{p,q}^{-1}$. In this case, let's take  $\eta_{q}: Stab_{\varphi} q \curvearrowright C_{q}$ defined by $\eta_{q}(h,\_) = \eta_{p}(h_{p,q}^{-1}h h_{p,q}(q),\_)$.

\begin{obs}Note that if  $q = p$, the new definition of $\eta_{p}$ coincides with the original one, since $h_{p,p} = 1$.
\end{obs}

\begin{prop}$\eta_{q}$ is an action by homeomorphisms.
\end{prop}

\begin{proof}We have that $\eta_{q}(h_{1}h_{2},\_) =  \eta_{p}(h_{p,q}^{-1}h_{1}h_{2}h_{p,q},\_) = \\ \eta_{p}(h_{p,q}^{-1}h_{1}h_{p,q}h_{p,q}^{-1}h_{2}h_{p,q},\_) = \eta_{p}(h_{p,q}^{-1}h_{1}h_{p,q},\_) \circ \eta_{p}(h_{p,q}^{-1}h_{2}h_{p,q},\_) = \eta_{q}(h_{1},\_) \circ  \eta_{q}(h_{2},\_)$, which implies that $\eta_{q}$ is a group action, which is by homeomorphisms because, by construction, it is equivalent to $\eta_{p}$.
\end{proof}

Let's take, for $p \in P'$ and $q \in Orb_{\varphi} p$, the compact $Stab_{\varphi} p +_{\partial_{q}}C_{q}$, where $\partial_{q} = (\partial_{p})_{\ast}$, is the pullback for  $\tau_{h_{p,q}^{-1}}: Stab_{\varphi}q \rightarrow Stab_{\varphi}p$, the conjugation map $\tau_{h_{p,q}^{-1}}(x) = h_{p,q}^{-1}xh_{p,q}$, and $id_{C_{p}}$. Denote by $L_{q},R_{q}: Stab_{\varphi} q \curvearrowright Stab_{\varphi} q$, respectively, the left and right multiplication actions.

\begin{prop}$L_{q}+\eta_{q}: Stab_{\varphi} q \curvearrowright Stab_{\varphi} q +_{\partial_{q}}C_{q}$ is an action by homeomorphisms.
\end{prop}

\begin{proof}We have that both diagrams commutes ($\forall h \in Stab_{\varphi} q$):

$$ \xymatrix{ Stab_{\varphi} q \ar[rr]_{L_{q}(h,\_)} \ar[d]_{\tau_{h_{p,q}^{-1}}} & & Stab_{\varphi} q \ar[d]^{\tau_{h_{p,q}^{-1}}} & C_{q} \ar[d]_{id} \ar[rr]_{\eta_{q}(h,\_)} & & C_{q} \ar[d]^{id} \\
            Stab_{\varphi} p \ar[rr]^{L_{p}(\tau_{h_{p,q}^{-1}}(h),\_)} & & Stab_{\varphi} p & C_{p} \ar[rr]^{\eta_{p}(\tau_{h_{p,q}^{-1}}(h),\_)} & & C_{p}} $$

In fact, on the first diagram both terms give us $h_{p,q}^{-1}hgh_{p,q}, \ \forall g \in Stab_{\varphi} q$. So, by \textrm{Proposition \ref{pullbackaction}}, we have that $L_{q}+\eta_{q}$ is an action by homeomorphisms.
\end{proof}

\begin{prop}$R_{q}+id: Stab_{\varphi} q \curvearrowright Stab_{\varphi} q +_{\partial_{q}}C_{q}$ is an action by homeomorphisms.
\end{prop}

\begin{proof}We have that both diagrams commutes ($\forall h \in Stab_{\varphi} q$):

$$ \xymatrix{ Stab_{\varphi} q \ar[rr]_{R_{q}(\_,h)} \ar[d]_{\tau_{h_{p,q}^{-1}}} & & Stab_{\varphi} q \ar[d]^{\tau_{h_{p,q}^{-1}}} & C_{q} \ar[d]_{id} \ar[r]_{id} & C_{q} \ar[d]^{id} \\
            Stab_{\varphi} p \ar[rr]^{R_{p}(\_,\tau_{h_{p,q}^{-1}}(h))} & & Stab_{\varphi} p & C_{p} \ar[r]^{id} & C_{p}} $$

In a fact, on the first diagram both terms gives us $h_{p,q}^{-1}ghh_{p,q}, \ \forall g \in Stab_{\varphi} q$. So, by \textrm{Proposition \ref{pullbackaction}}, we have that $R_{q}+id$ is an action by homeomorphisms.
\end{proof}

\begin{prop}\label{changeofcoordinates}The map $\tau_{h_{p,q}^{-1}}+id: Stab_{\varphi} q +_{\partial_{q}} C_{q} \rightarrow Stab_{\varphi} p +_{\partial_{p}} C_{p}$ is a homeomorphism.
\end{prop}

\begin{proof}This comes immediately from the pullback.
\end{proof}

\begin{prop}Let $g \in G$ such that $\varphi(g,p) = q$. Then, the application $\tau_{g} + \eta_{p}(h_{p,q}^{-1}g,\_): Stab_{\varphi} p +_{\partial_{p}} C_{p} \rightarrow Stab_{\varphi} q +_{\partial_{q}} C_{q}$ is a homeomorphism.
\end{prop}

\begin{proof}Let $h \in Stab_{\varphi} p$. We have that $ghg^{-1} = g h_{p,q}^{-1}h_{p,q}h h_{p,q}^{-1}h_{p,q}g^{-1}$. Since $g h_{p,q}^{-1}, h_{p,q}g^{-1}\in Stab_{\varphi}q$, we have that $\tau_{g} = R_{q}(\_,h_{p,q}g^{-1}) \circ L_{q} (g h_{p,q}^{-1},\_)\circ \tau_{h_{p,q}}$. Since $\eta_{q}(gh_{p,q}^{-1},\_) = \eta_{p}(h_{p,q}^{-1}gh_{p,q}^{-1}h_{p,q},\_) = \eta_{p}(h_{p,q}^{-1}g,\_)$, we have that $\tau_{g} + \eta_{p}(h_{p,q}^{-1}g,\_) = (R_{q}(\_,h_{p,q}g^{-1}) + id) \circ  (L_{q} (g h_{p,q}^{-1},\_) + \eta_{q}(gh_{p,q}^{-1},\_)) \circ (\tau_{h_{p,q}} + id)$, which is a homeomorphism because each term is a homeomorphism.
\end{proof}

\begin{cor}Let $q \in Orb_{\varphi} p$ and $g \in G$. Then $\tau_{g}+\eta_{p}(h_{p,\varphi(g,q)}^{-1}gh_{p,q},\_): Stab_{\varphi}q+_{\partial_{q}}C_{q} \rightarrow Stab_{\varphi}\varphi(g,q)+_{\partial_{\varphi(g,q)}}C_{\varphi(g,q)}$ is a homeomorphism.
\end{cor}

\begin{proof}We have that $\tau_{g} = \tau_{gh_{p,q}} \circ \tau_{h_{p,q}^{-1}}$ with $\varphi(gh_{p,q},p) = \varphi(g,q)$. We have also that  $\tau_{g}+\eta_{p}(h_{p,\varphi(g,q)}^{-1}gh_{p,q},\_) = (\tau_{gh_{p,q}} + \eta_{p}(h_{p,\varphi(g,q)}^{-1}gh_{p,q},\_) ) \circ (\tau_{h_{p,q}^{-1}} + id)$ which is a homeomorphism because each term is a homeomorphism.
\end{proof}

Since $P$ is the set of bounded parabolic points, we have that $\forall p \in P$,  $\varphi|_{Stab_{\varphi} p \times (Z - \{p\})}$ is proper and cocompact, so we are able to take the space $X_{p} = (Z - \{p\}) +_{(\partial_{p})_{\Pi}} C_{p}$ and we have that $\psi_{p} = \varphi|_{Stab_{\varphi} p\times (Z - \{p\})} + \eta_{p}: Stab_{\varphi}p \curvearrowright X_{p}$ is an action by homeomorphisms. We have also that the quotient map $\pi_{p}: X_{p} \rightarrow Z$, such that $\pi_{p}|_{Z - \{p\}}$ is the identity map and $\pi_{p}(C_{p}) = p$, is continuous and $Stab_{\varphi} p$-equivariant with respect to $\psi_{p}$ and $\varphi$, respectively.

Let's take, for $p\in P', q \in Orb_{\varphi} p$ and $g \in G$, $\psi_{q}(g,\_): X_{q} \rightarrow X_{\varphi(g,q)}$ given by $\psi_{q}(g,\_) = \varphi(g,\_)|_{Z - \{q\}} + \eta_{p}(h_{p,\varphi(g,q)}^{-1}g h_{p,q},\_)$, (since $h_{p,\varphi(g,q)}^{-1}g h_{p,q} \in Stab_{\varphi}p$).

\begin{obs}Observe that, for $g \in Stab_{\varphi}q$, this definition agrees with that one in the paragraph above.
\end{obs}

\begin{prop}$\psi_{q}(g,\_)$ is a homeomorphism.
\end{prop}

\begin{proof}We have that $\tau_{g}: Stab_{\varphi}q \rightarrow Stab_{\varphi}\varphi(g,q)$ such that $\tau_{g}(x) = gxg^{-1}$ is an isomorphism and the diagrams commutes ($\forall h \in Stab_{\varphi}q$):

$$ \xymatrix{ Z - \{q\} \ar[r]_{\varphi(h,\_)} \ar[d]_{\varphi(g,\_)} & Z -\{q\} \ar[d]_{\varphi(g,\_)} \\
            Z-\{\varphi(g,q)\} \ar[r]_{\varphi(\tau_{g}(h),\_)}  & Z-\{\varphi(g,q)\} } $$

So, the map $\varphi(g,\_)$ is $\tau_{g}$-equivariant. Since $\tau_{g}+ \eta_{p}(h_{p,\varphi(g,q)}^{-1}g h_{p,q},\_):  Stab_{\varphi} q \rightarrow Stab_{\varphi} \varphi(g,q)$ is continuous, we have, by \textbf{Proposition \ref{functor}}, that $\psi_{q}(g,\_)$ is continuous. Since $\psi_{q}(g,\_)^{-1} = \psi_{\varphi(g,q)}(g^{-1},\_)$, we have that it is continuous as well. Thus, $\psi_{q}(g,\_)$ is a homeomorphism.
\end{proof}

\begin{prop}$\forall p \in P', q \in Orb_{\varphi} p$ and $g,h \in G, \psi_{q}(gh,\_) = \psi_{\varphi(h,q)}(g,\_) \circ \psi_{q}(h,\_)$.
\end{prop}

\begin{proof}
We have that
$$\psi_{q}(gh,\_)|_{Z - \{q\}} = \varphi(gh,\_) = \varphi(g,\_) \circ \varphi(h,\_) = \psi_{\varphi(h,q)}(g,\_)|_{Z - \{\varphi(h,q)\}} \circ \psi_{q}(h,\_)|_{Z - \{q\}}$$

And
$$\psi_{q}(gh,\_)|_{C_{q}} = \eta_{p}(h_{p,\varphi(gh,q)}^{-1} g h h_{p,q},\_) = \eta_{p}(h_{p,\varphi(gh,q)}^{-1} g h_{p,\varphi(h,p)} h_{p,\varphi(h,p)}^{-1} h h_{p,q},\_) =$$
$$= \eta_{p}(h_{p,\varphi(gh,q)}^{-1} g h_{p,\varphi(h,p)},\_)\circ \eta_{p}(h_{p,\varphi(h,p)}^{-1} h h_{p,q},\_) = \psi_{\varphi(h,q)}(g,\_)|_{C_{\varphi(h,q)}} \circ \psi_{q}(h,\_)|_{C_{q}}$$

So $\psi_{p}(gh,\_) = \psi_{\varphi(h,q)}(g,\_) \circ \psi_{q}(h,\_)$. \end{proof}

And we have also that the diagram commutes:

$$ \xymatrix{ X_{p} \ar[d]_{\pi_{p}} \ar[r]_{\psi_{p}(g,\_)} & X_{\varphi(g,p)} \ar[d]^{\pi_{\varphi(g,p)}}  & \\
            Z \ar[r]^{\varphi(g,\_)}  & Z} $$

The induced map $\psi(g,\_): \prod\limits_{p \in P} X_{p} \rightarrow \prod\limits_{p \in P} X_{p}$ is continuous and, since $\psi_{p}(gh,\_) = \psi_{\varphi(h,p)}(g,\_) \circ \psi_{p}(h,\_)$, it follows that $\psi(gh,\_) = \psi(g,\_)\circ \psi(h,\_)$. Thus, $\psi$ is a group action of the group $G$. Since every map $\psi_{p}$ commutes the diagram above, we have that the subspace $X = \lim\limits_{\longleftarrow}(\{X_{p}\}_{p \in P},\{\pi_{p}\}_{p \in P})$ is $G$-invariant.

Let's take $\varpi_{p}: X \rightarrow X_{p}$ the projection map and $\pi = \pi_{p} \circ \varpi_{p}$ (it doesn't matter for which choice of the point $p$, because of the fiber product's definition). We call $\psi$ restrict to $X$ by $\varphi \ltimes \eta$ and, agreeing with the notation, we call $X$ by $Z \ltimes \mathcal{C}$. The pair $(Z \ltimes \mathcal{C}, \varphi\ltimes \eta)$ is called parabolic blowup of $(Z,\varphi)$ by  $(\mathcal{C},\eta)$. This construction depends on the minimal family $\{H_{p}\}_{p\in P'}$ but this dependence will be omitted through the text.

\subsection{Functoriality}

\begin{prop}\label{pullbackfunctor}Let $\{C'_{p}\}_{p\in P'}$ be another family of Hausdorff compact topological spaces, $\eta' = \{\eta'_{p}\}_{p\in P'}$ a family of actions by homeomorphisms $\eta'_{p}: Stab_{\varphi}p \curvearrowright C'_{p}$ and $\mathcal{C}' = \{Stab_{\varphi} p+_{\partial'_{p}}C'_{p}\}_{p\in P'}$ a family of equivariant perspective compactifications of the groups $Stab_{\varphi} p$ such that the actions $L_{p}+\eta'_{p}$ are by homeomorphisms. Consider also $\phi = \{\phi_{p}\}_{p\in P'}$ a family of continuous maps $\phi_{p}: C_{p} \rightarrow C'_{p}$ that are equivariant with respect to $\eta_{p}$ and $\eta'_{p}$, respectively, such that $id+\phi_{p}: Stab_{\varphi} p+_{\partial_{p}}C_{p} \rightarrow Stab_{\varphi} p+_{\partial'_{p}}C'_{p}$ are continuous. Then the maps $id+\phi_{p}:(Z - \{p\}) +_{(\partial_{p})_{\Pi}} C_{p} \rightarrow (Z - \{p\})+_{(\partial'_{p})_{\Pi}} C'_{p}$ induce the unique map $\varphi \ltimes \phi: Z \ltimes \mathcal{C} \rightarrow Z \ltimes \mathcal{C}'$ that is continuous and equivariant with respect to $\varphi\ltimes \eta$ and $\varphi \ltimes \eta'$, respectively, such that the following diagram commutes ($\forall p \in P'$):

$$ \xymatrix{ Z \ltimes \mathcal{C}  \ar[d]_{\varpi_{p}} \ar[r]_{\varphi \ltimes \phi} & Z \ltimes \mathcal{C}'  \ar[d]^{\varpi'_{p}}  & \\
            X_{p} \ar[r]_{id+\phi_{p}}  & X'_{p}} $$

\end{prop}

\begin{proof}Let $p\in P'$ and $q \in Orb_{\varphi} p$. Let $X'_{q} = (Z -  \{q\})\!+_{(\partial'_{q})_{\Pi}} C'_{q}$ and $\phi_{q}: C_{q} \rightarrow C'_{q}$ given by $\phi_{q} = \phi_{p}$. The map $id \! + \! \phi_{q}: Stab_{\varphi} q +_{\partial_{q}} C_{q} \rightarrow Stab_{\varphi} q +_{\partial'_{q}} C'_{q}$ is continuous, since $id \! + \! \phi_{q} = (\tau_{h_{p,q}} \! \! + \! id) \circ (id \! + \! \phi_{p}) \circ (\tau_{h_{p,q}^{-1}} \! + \! id)$ and all of the terms are continuous. We have that $id_{Stab_{\varphi} q}$ is $L_{q}$-equivariant and $id \! + \! \phi_{q}$ is continuous, which implies that $id \! + \! \phi_{q}:(Z \! - \! \{q\}) \! +_{(\partial_{q})_{\Pi}} C_{q} \rightarrow (Z \! - \! \{q\})+_{(\partial'_{q})_{\Pi}} C'_{q}$ is continuous as well. By the definition of the maps, the diagram always commutes:

$$ \xymatrix{ X_{q} \ar[d]_{\pi_{q}} \ar[r]_{id+\phi_{q}} & X'_{q} \ar[d]^{\pi'_{q}}  & \\
            Z \ar[r]^{id}  & Z} $$

Where $\pi_{q}$ and $\pi'_{q}$ are the quotient maps. So they induce a continuous map $\varphi \ltimes \phi: Z \ltimes \mathcal{C} \rightarrow Z \ltimes \mathcal{C}'$ that commutes the diagrams:

$$ \xymatrix{ Z \ltimes \mathcal{C}  \ar[d]_{\varpi_{q}} \ar[r]_{\varphi \ltimes \phi} & Z \ltimes \mathcal{C}'  \ar[d]^{\varpi'_{q}}  & \\
            X_{q} \ar[r]_{id+\phi_{q}}  & X'_{q}} $$

Where $\varpi_{q}$ and $\varpi'_{q}$ are the projection maps. Let's consider the following diagram:

$$ \xymatrix{ Z \ltimes \mathcal{C}  \ar[dd]_{\varpi_{q}} \ar[rr]_{\varphi \ltimes \phi} \ar[rd]_{\varphi\ltimes\eta(g,\_)} \ar@{}[rrrd]|{\textbf{1}} & & Z \ltimes \mathcal{C}'  \ar[dd]_<<{\varpi'_{q}} \ar[rd]^{\varphi\ltimes\eta'(g,\_)} & \\
            & Z \ltimes \mathcal{C}  \ar[dd]_<<{\varpi_{\varphi(g,q)}} \ar[rr]_<<{ \ \ \ \ \ \ \ \varphi \ltimes \phi} & & Z \ltimes \mathcal{C}'  \ar[dd]^{\varpi'_{\varphi(g,q)}} \\
            X_{q} \ar[dd]_{\pi_{q}} \ar[rr]_{ \ \ \ \ \ \ \ \ \ \ id+\phi_{q}} \ar[rd]_{\psi_{q}(g,\_)} \ar@{}[rrrd]|{\textbf{2}} & & X'_{q} \ar[dd]_<<{\pi'_{q}} \ar[rd]^{\psi'_{q}(g,\_)} & \\
            & X_{\varphi(g,q)} \ar[dd]_<<{\pi_{\varphi(g,q)}} \ar[rr]_<<{ \ \ \ \ \ \ \ \ \ \ id+\phi_{\varphi(g,q)}} & & X'_{\varphi(g,q)} \ar[dd]^{\pi'_{\varphi(g,q)}} & \\
            Z \ar[rr]^{ \ \ \ \ \ id} \ar[rd]_{\varphi(g,\_)} \ar@{}[rrrd]|{\textbf{3}} & & Z \ar[rd]^{\varphi(g,\_)} & \\
            & Z \ar[rr]^{id} & & Z } $$

Where $\psi'_{q}(g,\_) = \varphi(g,\_)|_{Z-\{q\}}+\eta'_{p}(h_{p,\varphi(g,q)}^{-1}gh_{p,q})$. All squares commute, except, a priori, the squares $\textbf{1}$ and $\textbf{2}$. We have that $(id+\phi_{\varphi(g,q)}) \circ \psi_{q}(g,\_)|_{Z - \{q\}} = \varphi(g,\_)|_{Z - \{q\}} = \psi'_{q}(g,\_) \circ (id+\phi_{q})|_{Z - \{q\}}$ and $(id+\phi_{\varphi(g,q)}) \circ \psi_{q}(g,\_)|_{C_{q}} = \phi_{\varphi(g,q)} \circ \eta_{p}(h_{p,\varphi(g,q)}^{-1}gh_{p,q},\_) = \eta'_{p}(h_{p,\varphi(g,q)}^{-1}gh_{p,q},\_) \circ \phi_{\varphi(g,q)} = \eta'_{p}(h_{p,\varphi(g,q)}^{-1}gh_{p,q},\_) \circ \phi_{q} = \psi'_{q}(g,\_) \circ (id+\phi_{q})|_{C_{q}}$, since $\phi_{p}$ is equivariant with respect to $\eta_{p}$ and $\eta'_{p}$ and $\phi_{p} = \phi_{q} = \phi_{\varphi(g,q)}$. So square number $\textbf{2}$ commutes. We have that $\varphi\ltimes\eta(g,\_)$ is induced by the maps $\psi_{q}(g,\_)$ and $\varphi(g,\_)$, $\forall q \in P$ and $\varphi\ltimes \phi$ is induced by the maps $id_{Z-\{q\}}+\phi_{q}$ and $id_{Z}, \ \forall q \in P$. Because the square number $\textbf{2}$ commutes and by the functoriality of the induced maps of the inverse limit, it follows that the square number $\textbf{1}$ commutes.

Thus, $\varphi\ltimes \phi$ is equivariant with respect to $\varphi\ltimes \eta$ and $\varphi\ltimes \eta'$, respectively.

Let's suppose that there exists another equivariant and continuous map $\chi: Z \ltimes \mathcal{C} \rightarrow Z \ltimes \mathcal{C}'$ that is continuous and equivariant with respect to $\varphi\ltimes \eta$ and $\varphi \ltimes \eta'$, respectively, such that commutes the following diagram ($\forall p \in P'$):

$$ \xymatrix{ Z \ltimes \mathcal{C} \ar[d]_{\varpi_{p}} \ar[r]_{\chi} & Z \ltimes \mathcal{C}'  \ar[d]^{\varpi'_{p}}  \\
            X_{p} \ar[r]_{id+\phi_{p}}  & X'_{p}} $$

Let $q \in Orb_{\varphi} p$ and $g \in G$ such that $\varphi(g,p) = q$. We have that the following diagram commutes:

$$ \xymatrix{  Z \ltimes \mathcal{C}  \ar[d]_{\varpi_{q}} \ar[r]_{\varphi \ltimes \eta(g^{-1},\_)}  & Z \ltimes \mathcal{C}  \ar[d]_{\varpi_{p}} \ar[r]_{\chi} & Z \ltimes \mathcal{C}'  \ar[d]^{\varpi'_{p}} \ar[r]_{\varphi \ltimes \eta'(g,\_)} &  Z \ltimes \mathcal{C}' \ar[d]^{\varpi'_{q}}   \\
            X_{q} \ar[r]_{\psi_{q}(g^{-1},\_)}  & X_{p} \ar[r]_{id+\phi_{p}}   & X'_{p} \ar[r]_{\psi'_{q}(g,\_)} & X'_{q}} $$

Since $\chi$ is equivariant, we have that $\varphi \ltimes \eta'(g,\_) \circ \chi \circ \varphi \ltimes \eta(g^{-1},\_) = \chi$. We have also that $\psi'_{q}(g,\_) \circ (id+\phi_{p}) \circ \psi_{q}(g^{-1},\_)|_{Z-\{q\}} = \varphi(g,\_)\circ \varphi(g^{-1},\_)|_{Z-\{q\}} = id_{Z-\{q\}}$ and $\psi'_{q}(g,\_) \circ (id+\phi_{p}) \circ \psi_{q}(g^{-1},\_)|_{C_{q}} = \eta'_{p}(h_{p,q}^{-1}g,\_) \circ \phi_{p} \circ \eta_{p}(g^{-1}h_{p,q},\_) = \phi_{p} = \phi_{q}$, since $\phi_{p}$ is equivariant and $\phi_{p} = \phi_{q}$. So $\psi'_{q}(g,\_) \circ (id+\phi_{p}) \circ \psi_{q}(g^{-1},\_) = id+\phi_{q}$. Then the diagram commutes $\forall q \in P$:

$$ \xymatrix{ Z \ltimes \mathcal{C} \ar[d]_{\varpi_{q}} \ar[r]_{\chi} & Z \ltimes \mathcal{C}'  \ar[d]^{\varpi'_{q}}  \\
            X_{q} \ar[r]_{id+\phi_{q}}  & X'_{q}} $$

By the universal property of the pullback, we have that $\chi = \varphi\ltimes \phi$.
\end{proof}

\begin{defi}Let $G$ be a group. Let's denote by $Act(G)$ the category where the objects are  actions by homeomorphisms of $G$ on Hausdorff compact topological spaces and the morphisms are equivariant continuous maps. Let $MAct(G)$ be the full subcategory of $Act(G)$ such that the objects are actions on metrizable spaces.
\end{defi}

\begin{prop}\label{fuctoriality}Let $G$ be a group, $Z$ a Hausdorff compact space, $\varphi: G \curvearrowright Z$ an action by homeomorphisms, $P$ the set of bounded parabolic points of $\varphi$, $P' \subseteq P$ a representative set of orbits and $H = \{H_{p}\}_{p \in P'}$ a family of minimal subsets of $G$ such that $Orb_{\varphi|_{H_{p}\times Z}} p = Orb_{\varphi} p$. Then the map $\varphi\ltimes: \prod_{p\in P'}EPers(Stab_{\varphi}p) \rightarrow Act(G)$ such that $\varphi\ltimes(\breve{\eta}) = \varphi\ltimes\eta$ and $\varphi\ltimes(\phi) = \varphi\ltimes \phi$, where $\breve{\eta} = \{Stab_{\varphi}p+_{\partial_{p}}C_{p}\}_{p\in P'}$ is a family of equivariant perspective compactifications with respective actions $\eta = \{\eta_{p}\}_{p\in P'}$ and $\phi$ is a family of equivariant and continuous maps, is a functor. \eod
\end{prop}

Let's denote $\varphi\ltimes$ the parabolic blowup functor of $\varphi$.

\begin{prop}If $P$ is countable and $Z$ is metrizable, then the functor $\varphi\ltimes$  restricts itself to the functor (which we are maintaining the same name) $\varphi\ltimes: \prod_{p\in P'}EMPers(Stab_{\varphi}p) \rightarrow MAct(G)$.
\end{prop}

\begin{proof}We have that $X_{p}$ is Hausdorff and, because $Z-\{p\}$ and $C_{p}$ have both countable basis, $X_{p}$ is metrizable by \textbf{Proposition \ref{uniaometrizavel}}. This implies that $\prod_{p\in P} X_{p}$ is metrizable, since $P$ is countable. Since $X$ is a subspace of $\prod_{p\in P} X_{p}$, it follows that $X$ is metrizable.
\end{proof}

Now we are going to show that the parabolic blowup do not depend of the choice of the family $H$:

\begin{prop}Let $\varphi\ltimes'$ be induced by the map $\varphi$, by $P'$ and by the minimal family $H' = \{H'_{p}\}_{p\in P'}$. So, there exists a natural isomorphism between $\varphi\ltimes$ and $\varphi\ltimes'$.
\end{prop}

\begin{proof}Let $X = Z \ltimes \mathcal{C}$ and $X' = Z\ltimes' \mathcal{C}$ for the same family of compactifications $\breve{\eta}$ and respective actions $\eta$ but minimal families $H$ and $H'$, respectively. We have, for $p \in P'$ and $q \in Orb_{\varphi} p$, that $\varphi(g,\_)+\eta_{p}(gh_{p,q},\_): X_{q} \rightarrow X_{p}$, $id: X_{p} \rightarrow X'_{p}$ and $\varphi(g^{-1},\_)+\eta_{p}(h_{p,q}'^{-1}g^{-1},\_): X'_{p} \rightarrow X'_{q}$ are continuous, where $g \in G$ is such that $\varphi(g,q) = p$ and $X'_{q}$ is constructed from the family $H'$. Let's take $T_{\eta q}: X_{q} \rightarrow X_{q}'$ defined by $id_{Z-\{q\}} + \eta_{p}(h_{p,q}'^{-1}h_{p,q},\_)$ (observe that $h_{p,q}'^{-1}h_{p,q} \in Stab_{\varphi} p$). We have that $T_{\eta q} = (\varphi(g^{-1},\_)+\eta_{p}(h_{p,q}'^{-1}g^{-1},\_)) \circ (\varphi(g,\_)+\eta_{p}(gh_{p,q},\_))$, which implies that it is continuous and therefore a homeomorphism.

We have that, $\forall q \in P$, the diagram commutes:

$$ \xymatrix{ X_{q} \ar[d]_{\pi_{q}} \ar[r]_{T_{\eta q}} & X_{q}' \ar[d]^{\pi_{q}'} \\
            Z \ar[r]_{id}  & Z} $$

This implies that the family of maps $\{T_{\eta q}\}_{q\in P}$ induces a homeomorphism $T_{\eta}: X \rightarrow X'$. And it follows that $T_{\eta}$ is $G$ - equivariant, because $\forall g \in G$, the diagrams commute:

$$ \xymatrix{ Z - \{q\} \ar[d]_{id} \ar[r]_{\varphi(g,q)} & Z -\{\varphi(g,q)\} \ar[d]^{id}  & & C_{q} \ar[d]^{\eta_{p}(h_{p,q}'^{-1} h_{p,q},\_)} \ar[rrr]^{\eta_{p}(h_{p,\varphi(g,q)}^{-1}gh_{p,q},\_)} & & & C_{\varphi(g,q)} \ar[d]^{\eta_{p}(h_{p,\varphi(g,q)}'^{-1}h_{p,\varphi(g,q)},\_)} \\
            Z-\{q\} \ar[r]_{\varphi(g,q)}  & Z-\{\varphi(g,q)\} & & C_{q} \ar[rrr]_{\eta_{p}(h_{p,\varphi(g,q)}'^{-1}gh_{p,q}',\_)} & & & C_{\varphi(g,q)}} $$

So $T_{\eta}$ is an isomorphism between $\varphi\ltimes \eta$ and $\varphi\ltimes'\eta$. Take the family of maps $T = \{T_{\eta}\}_{\eta}: \varphi\ltimes \Rightarrow \varphi\ltimes'$. Let's take $Y = Z \times \mathcal{D}$ and $Y' = Z \times' \mathcal{D}$, for a family of spaces $ \{D_{p}\}_{p\in P'}$ and a family of actions $\mu = \{\mu_{p}\}_{p\in P'}$ with a equivariant perspective family of compact spaces $\mathcal{D} = \{Stab_{\varphi}p+_{\delta_{p}}D_{p}\}_{p\in P'}$ and constructed from minimal families $H$ and $H'$, respectively. Let $\phi = \{\phi_{p}\}_{p\in P'}: \eta \rightarrow \mu$ be a morphism. We have that, $\forall p \in P'$ and $q \in Orb_{\varphi} p,$ the diagrams commute (because $\phi_{q}$ is equivariant with respect to $\eta_{q}$ and $\mu_{q}$):

$$ \xymatrix{ Z - \{q\} \ar[d]_{id} \ar[r]_{id} & Z -\{q\} \ar[d]^{id}  & & C_{q} \ar[d]^{T_{\eta q}} \ar[r]^{\phi_{q}} & D_{q} \ar[d]^{T_{\mu q}} \\
            Z-\{q\} \ar[r]_{id}  & Z-\{q\} & & C_{q} \ar[r]_{\phi_{q}} & D_{q}} $$

Remember that $\phi_{q}$ is defined equal to $\phi_{p}$ in \textbf{Proposition \ref{pullbackfunctor}}. This implies that the diagram commutes:

$$ \xymatrix{ X \ar[d]_{T_{\eta}} \ar[r]_{\varphi\ltimes \phi} & Y \ar[d]^{T_{\mu}} \\
            X' \ar[r]_{\varphi\ltimes' \phi}  & Y'} $$

Thus, $T$ is a natural transformation. Since, $\forall \eta, \ T_{\eta}$ is an isomorphism, it follows that $T$ is a natural  isomorphism.
\end{proof}

\subsection{A perspectivity for the parabolic blowup}

Our goal here is, given $Z = G+_{\partial}Y \in EPers(G)$, to figure that the parabolic blowup is also perspective.

Let $G$ be a group, $G+_{\partial}Y \in EPers(G)$ a compact Hausdorff space, $L+\varphi: G \curvearrowright G+_{\partial}Y$ an action by homeomorphisms, $P \subseteq Y$ the set of bounded parabolic points of $L+\varphi$, $P' \subseteq P$ a representative set of orbits, $\mathcal{C} = \{C_{p}\}_{p\in P'}$ a family of compact Hausdorff spaces, $H = \{H_{p}\}_{p\in P'}$, with $H_{p} \subseteq G$ minimal sets such that $1\in H_{p}$ and $Orb_{\varphi|_{H_{p}\times Z}} p = Orb_{\varphi} p$ and $\{Stab_{\varphi} p+_{\partial_{p}}C_{p}\}_{p\in P'}$ a family of spaces with the equivariant perspective property with actions $\eta = \{\eta_{p}\}_{p\in P'}$, such that $L_{p}+\eta_{p}: Stab_{\varphi} p \curvearrowright Stab_{\varphi} p+_{\partial_{p}}C_{p}$ is by homeomorphisms.

\begin{obs}The set of bounded parabolic points of $L+\varphi$ is a subset of the set of bounded parabolic points of $\varphi$. We show that in some cases those sets coincide (\textbf{Proposition \ref{parabolicpoints}}).
\end{obs}

\begin{prop}\label{inclusaogrupo}The map $\iota: G \rightarrow (G+_{\partial}Y)\ltimes \mathcal{C}$, induced by the inclusion maps of $\iota': G \rightarrow G+_{\partial}Y$ and $\iota_{p}: G \rightarrow X_{p}$, is an embedding, $Im \ \iota$ is open and $\iota$ is equivariant with respect to $L$ and $(L+\varphi)\ltimes \eta$.
\end{prop}

\begin{proof}It follows by \textbf{Proposition \ref{embedding}} that $\iota$ is an embedding.

We have that $G$ is open in $G+_{\partial}Y$, which implies that $\pi^{-1}(G)$ is open in $(G+_{\partial}Y)\ltimes \mathcal{C}$ (where $\pi$ is the projection). Since $\pi^{-1}(G) = \iota(G)$, because $\pi \circ\iota = \iota'$ and $\pi|_{\pi^{-1}(G)}$ is injective, we have that $\iota(G)$ is open in $(G+_{\partial}Y)\ltimes \mathcal{C}$.

Let's consider the diagram (for $g \in G$):

$$ \xymatrix{ G \ar[rd]_{\iota'} \ar[rr]^{\iota} \ar[dd]_{L(g,\_)} & & (G+_{\partial}Y)\ltimes \mathcal{C} \ar[ld]^{\pi} \ar[dd]^{(L+\varphi)\ltimes \eta(g,\_)} \\
            & G+_{\partial}Y \ar[dd]^<<{L+\varphi(g,\_)} & \\
            G \ar[rr]^{\iota} \ar[rd]^{\iota'}& & (G+_{\partial}Y)\ltimes \mathcal{C} \ar[ld]^{\pi} \\
            & G+_{\partial}Y &} $$

Since $\iota'$ and $\pi$ are equivariant, we have that both parallelograms commute. Since $\pi|_{\pi^{-1}(G)} = \pi|_{\pi^{-1}(\iota'(G))}$ is injective, we have that the rectangle commutes. Thus, $\iota$ is equivariant.
\end{proof}

So $(G+_{\partial}Y)\ltimes \mathcal{C}$ can be identified with an object in $T_{2}EComp(G)$.

\begin{prop}Let's take a family of spaces $\mathcal{D} = \{D_{p}\}_{p\in P'}$, a family of actions $\mu = \{\mu_{p}\}_{p\in P'}$ with a family of compact spaces $\check{\mu} = \{Stab_{\varphi}p+_{\delta_{p}}D_{p}\}_{p\in P'}$ with the equivariant perspective property and $Y = (G+_{\partial}Y)\ltimes \mathcal{D}$. Let $\phi = \{\phi_{p}\}_{p\in P'}: \check{\eta} \rightarrow \check{\mu}$ be a morphism. Then the diagram commutes:

$$ \xymatrix{ G \ar[rd]_{\iota_{2}} \ar[r]^<<{\ \ \ \ \ \iota_{1}} & (G+_{\partial}Y) \ltimes \mathcal{C} \ar[d]^{(L+\varphi) \ltimes \phi} \\
            & (G+_{\partial}Y) \ltimes \mathcal{D}} $$

Where $\iota_{1}$ and $\iota_{2}$ are the inclusion maps.

\end{prop}

\begin{proof}\textbf{Proposition \ref{invarianceofnonparabolics}}.
\end{proof}

So every morphism of the form $(L+\varphi) \ltimes \phi$ can be seen as the identity on $G$, which implies that it is a morphism in $T_{2}EComp(G)$. Thus, our parabolic blowup functor induces a new functor $(\varphi,\partial)\ltimes: \prod_{p\in P'}EPers(Stab_{\varphi} p) \rightarrow T_{2}EComp(G)$ that commutes the diagram:

$$ \xymatrix{ & Act(G) \\
            \prod\limits_{p\in \tilde{P'}}EPers(Stab_{\varphi} p) \ar[r]^{ \ \ \ (\varphi,\partial) \ltimes} \ar[ru]^{(L+\varphi) \ltimes} & T_{2}EComp(G) \ar[u]^{\Gamma} } $$

Where $\Gamma: T_{2}EComp(G) \rightarrow Act(G)$ is the forgetful functor.

\begin{prop}\label{persppullback}If $G+_{\partial}Y \in EPers$, then $(G+_{\partial}Y) \ltimes \mathcal{C} \in EPers(G)$.
\end{prop}

\begin{proof}We already have that $(G+_{\partial}Y) \ltimes \mathcal{C}$ is Hausdorff.

Since $G+_{\partial}Y$ has the perspectivity property, we have that the action $R+id: G \curvearrowright G+_{\partial}Y$ is by homeomorphisms. We have also that $\forall p \in P$, $R+id: Stab_{\varphi}p \curvearrowright Stab_{\varphi}p+_{\partial_{p}}C_{p}$ is by homeomorphisms.

Let $g \in G$. We have that $\forall p \in P$, $R_{p}(\_,g) = (R(\_,g)+id)|_{(G+_{\partial}Y) - \{p\}}$ is continuous and $(L+\varphi)$-equivariant (since $L$ and $R$ commute). So, by \textbf{Proposition \ref{functor}}, $R_{p}(\_,g)+id: X_{p} \rightarrow X_{p}$ is continuous, where $X_{p} = (G+_{\partial}Y) - \{p\})+_{(\partial_{p})_{\Pi}}C_{p}$, and commute the diagram:

$$ \xymatrix{ X_{p} \ar[r]_{R_{p}(\_,g)+id} \ar[d]_{\pi_{p}} & X_{p} \ar[d]_{\pi_{p}} \\
             G+_{\partial}Y \ar[r]_{R(\_,g)+id} & G+_{\partial}Y } $$

So the maps $R(\_,g)+id$ and $\{R_{p}(\_,g)+id\}_{p\in P}$ induce a homeomorphism $\tilde{R}_{g}: (G+_{\partial}Y) \ltimes \mathcal{C} \rightarrow (G+_{\partial}Y) \ltimes \mathcal{C}$ that commutes the diagrams:

$$ \xymatrix{ (G+_{\partial}Y) \ltimes \mathcal{C} \ar[r]_{\tilde{R}_{g}} \ar[d]_{\varpi_{p}} & (G+_{\partial}Y) \ltimes \mathcal{C} \ar[d]_{\varpi_{p}} & & (G+_{\partial}Y) \ltimes \mathcal{C} \ar[r]_{\tilde{R}_{g}} \ar[d]_{\pi} & (G+_{\partial}Y) \ltimes \mathcal{C} \ar[d]_{\pi} \\
             X_{p} \ar[r]_{R_{p}(\_,g)+id} & X_{p} & & G+_{\partial}Y \ar[r]_{R(\_,g)+id} & G+_{\partial}Y }$$

Since $R(\_,g)+id$ is bijective and the second diagram commute, we can decompose $\tilde{R}_{g} = \alpha+\beta: \pi^{-1}(G)+_{f}\pi^{-1}(Y) \rightarrow \pi^{-1}(G)+_{f}\pi^{-1}(Y)$, with the appropriate $f$. To say that $\alpha$ corresponds to the right multiplication for $g$ in the copy of $G$ in $(G+_{\partial}Y) \ltimes \mathcal{C}$ is to say that the upper trapezium of the diagram below commutes.

$$ \xymatrix{ G \ar[rrr]^{R(\_,g)} \ar[rd]^{\iota} \ar[dd]^{\iota'} & & & G \ar[ld]^{\iota} \ar[dd]^{\iota'} \\
            & (G+_{\partial}Y) \ltimes \mathcal{C} \ar[r]_{\tilde{R}_{g}} \ar[ld]_{\pi} & (G+_{\partial}Y) \ltimes \mathcal{C} \ar[rd]_{\pi} & \\
             G+_{\partial}Y \ar[rrr]_{R(\_,g)+id} & & & G+_{\partial}Y }$$

The triangles, the other trapezium and the outside rectangle are commutative. So $\pi \circ \iota \circ R(\_,g) = \iota' \circ R(\_,g) = (R(\_,g) + id)\circ \iota' = (R(\_,g) + id) \circ \pi \circ \iota = \pi \circ \tilde{R}_{g} \circ \iota$. Since $\pi|_{\pi^{-1}(G)}$ is injective, $\pi \circ \iota \circ R(\_,g)(G) \subseteq \iota'(G)$ and $\pi \circ \tilde{R}_{g} \circ \iota(G) \subseteq \iota'(G)$, it follows that $\iota \circ R(\_,g) = \tilde{R}_{g} \circ \iota$. So $\alpha$ is the right multiplication by the element $g$.

Let's take $x \in \pi^{-1}(Y)$. If $\pi(x) \notin P$, then $\#\pi^{-1}(\pi(x)) = 1$ and $\pi \circ \tilde{R}_{g}(x) = \pi(x)$, which implies that $\tilde{R}_{g}(x) = x$. If $\pi(x) = p \in P$, then $\#\varpi_{p}^{-1}(\varpi_{p}(x)) = 1$ and $\varpi_{p} \circ \tilde{R}_{g}(x) = \varpi_{p}(x)$, which implies that $\tilde{R}_{g}(x) = x$. So $\beta = id$.

So $\tilde{R}_{g}$ is the right multiplication map of $g$ when restricted to $\pi^{-1}(G)$ and the identity map everywhere else. Then, there is an action by homeomorphisms $\tilde{R}: G \curvearrowright (G+_{\partial}Y) \ltimes \mathcal{C}$ (defining $\tilde{R}(\_,g) = \tilde{R}_{g}$) which is the right multiplication in $\pi^{-1}(G)$ and the trivial action everywhere else. Thus, $(G+_{\partial}Y) \ltimes \mathcal{C}$ has the perspectivity property.
\end{proof}

So we have that if $G+_{\partial}Y$ has the equivariant perspectivity property, then our new functor restricts to $(\varphi,\partial)\ltimes: \prod_{p\in P'}EPers(Stab_{\varphi} p) \rightarrow EPers(G)$.

\begin{prop}If $G$ is dense in $G+_{\partial}Y$ and, $\forall p \in P'$, $Stab_{\varphi}p$ is dense in $Stab_{\varphi}p+_{\partial_{p}}C_{p}$, then $\iota(G)$ is dense in $(G+_{\partial}Y)\ltimes \mathcal{C}$.
\end{prop}

\begin{proof}Let $x \in (G+_{\partial}Y)\ltimes \mathcal{C} - \iota(G)$.

Let's suppose that $\pi(x) \in Y-P$. Then there exists a net $\{g_{\gamma}\}_{\gamma \in \Gamma}$ in $G$ that converges to $\pi(x)$. Since $\#\pi^{-1}(\pi(x)) = 1$, take a net $\{x_{\gamma}\}_{\gamma \in \Gamma} \subseteq X$ such that $\forall \gamma \in \Gamma$, $\pi(x_{\gamma}) = g_{\gamma}$. We have that this net is contained in $\iota(G)$ and converges to $x$. So $x$ is in the closure of $\iota(G)$.

Let's suppose that $\pi(x) = p \in P$. Then there exists a net $\{g_{\gamma}\}_{\gamma\in \Gamma}$ in $Stab_{\varphi}p$ that converges to $\varpi_{p}(x)$. Since $\#\varpi^{-1}_{p}(\varpi_{p}(x)) = 1$, take a net $\{x_{\gamma}\}_{\gamma \in \Gamma} \subseteq X$ such that $\forall \gamma \in \Gamma$, $\varpi_{p}(x_{\gamma}) = g_{\gamma}$. We have that this net is contained in $\iota(G)$ and converges to $x$. So $x$ is in the closure of $\iota(G)$.

Thus $\iota(G)$ is dense in $(G+_{\partial}Y)\ltimes \mathcal{C}$.
\end{proof}

\subsection{Dynamic quasiconvexity}

\begin{defi}Let $\psi : G \curvearrowright X$ be an action by homeomorphisms, where $X$ is a Hausdorff compact space. Let $K$ be a closed subset of $X$ and $H = \{g \in G: \psi(g,K) = K\}$. We say that $K$ is $\psi$-quasiconvex if $\forall u \in \U$, the set $\{g H:  \psi(g,K) \notin Small(u)\}$ is finite, where $\U$ is the uniform structure compatible with the topology of $X$.

Let $G+_{\partial}Y\in EPers(G)$ and $H < G$. We say that $H$ is dynamically quasiconvex if $\forall u \in \U_{\partial}$, $\#\{g H: \partial (gH) \notin Small(u)\} < \aleph_{0}$, where $\U_{\partial}$ is the uniform structure compatible with the topology of $G+_{\partial}Y$.
\end{defi}

\begin{obs}Let $L+\psi: G \curvearrowright G+_{\partial}Y$ be an action with the perspectivity property and $H$ a subgroup of $G$. If $\forall g \notin H$, $\partial(gH) \neq \partial(H)$, then $H = \{g \in G: \psi(g,\partial(H)) = \partial(H)\}$, which implies that  $\psi$-quasiconvexity of $\partial(H)$ is equivalent to dynamic quasiconvexity of $H$.

If $\psi$ has the convergence property, then this definition of dynamic quasiconvexity coincides with the usual one (i.e. for every $u \in \U_{\partial}$, the set $\{g H: \Lambda gH \notin Small(u)\}$ is finite).
\end{obs}

\begin{prop}\label{convdimtopsub}Let $\psi : G \curvearrowright X$ be an action by homeomorphisms, where $X$ is a Hausdorff compact space. Let $K$ be a closed subset of $X$ and $H = \{g \in G: \psi(g,K) = K\}$. Suppose that $\forall g_{1},g_{2} \in G$, $\psi(g_{1},K)\cap \psi(g_{2},K) = \emptyset$ or $\psi(g_{1},K) = \psi(g_{2},K)$. We denote by $\sim = \Delta^{2} X\cup \bigcup_{g\in G} \psi(g,K)^{2}$.  Then $K$ is dynamically quasiconvex if and only if $\sim$ is topologically quasiconvex.

\end{prop}

\begin{proof}$(\Rightarrow)$ Let $\U$ be the uniform structure compatible with the topology of $X$ and $u \in \U$. We define $f_{u}:\{g H: \psi(g,K) \notin Small(u)\} \rightarrow \{\psi(g,K): \psi(g,K) \notin Small(u)\}$ by $f_{u}(gH) = \psi(g,K)$. By construction this map is surjective. If $K$ is $\psi$-quasiconvex, then  $\forall u \in \U$, $\#\{g H: \psi(g,K) \notin Small(u)\} < \aleph_{0}$, which implies that $\forall u \in \U$, $\#\{\psi(g,K): \psi(g,K) \notin Small(u)\} < \aleph_{0}$. Thus, $\sim$ is topologically quasiconvex.

$(\Leftarrow)$ Let $g_{1},g_{2} \in G$ such that $\psi(g_{1},K) = \psi(g_{2},K)$. Then $\psi(g_{1}^{-1}g_{2},K) = \psi(g_{1}^{-1},\psi(g_{2},K)) = \psi(g_{1}^{-1},\psi(g_{1},K)) = \psi(g_{1}^{-1}g_{1},K) = K$. Then $g_{1}^{-1}g_{2} \in H$, which implies that $g_{2}H = g_{1}H$. So $f_{u}$ is injective, which implies that it is bijective. Thus, $\forall u \in \U$,  $\#\{g H: \psi(g,K) \notin Small(u)\} = \#\{\psi(g,K): \psi(g,K) \notin Small(u)\}$, which implies that $\sim$ topologically quasiconvex implies $K$  $\psi$-quasiconvex.
\end{proof}

\begin{prop}\label{convdimtop}Let $G+_{\partial}Y\in EPers(G)$, $\U_{\partial}$ the uniform structure compatible with $G+_{\partial}Y$ and $H < G$. Suppose that $\forall g_{1},g_{2} \in G$, $\partial(g_{1}H)\cap \partial(g_{2}H) = \emptyset$ or $\partial(g_{1}H) = \partial(g_{2}H)$. We denote by $\sim = \Delta^{2} (G+_{\partial}Y)\cup \bigcup_{g\in G} \partial(gH)^{2}$.  If $H$ is dynamically quasiconvex, then $\sim$ is topologically quasiconvex. If $\forall g \notin H$, $\partial(gH) \neq \partial(H)$, then the converse holds.

\end{prop}

\begin{proof}$(\Rightarrow)$ Analogous to the last proposition.

$(\Leftarrow)$ Let $L+\psi: G \curvearrowright G+_{\partial}Y$ be an action with the perspectivity property. Since $\forall g \notin H$, $\partial(gH) \neq \partial(H)$, then $\psi$-quasiconvexity of $\partial(H)$ is equivalent to dynamic quasiconvexity of $H$. Thus, the proposition follows from the previous one.
\end{proof}

\begin{cor}Let $G+_{\partial}Y\in EPers(G)$ and $H$ a dynamically quasiconvex subgroup of $G$ such that $\forall g_{1}\neq g_{2}\in G$, $\partial(g_{1}H) \cap \partial(g_{2}H) = \emptyset$ or $\partial(g_{1}H) = \partial(g_{2}H)$. If $A \subseteq \{gH: g \in G\}$ and $\sim_{A} = \Delta^{2} X \cup \bigcup\limits_{gH \in A}\partial(gH)^{2}$, then $X/ \sim_{A}$ is Hausdorff.
\end{cor}

\begin{proof}Immediate from \textbf{Proposition \ref{quaseconvexidadetop}}.
\end{proof}

\begin{lema}Let $G+_{\partial}Y$ be a Hausdorff compact space and $\psi: G \curvearrowright G+_{\partial}Y$ an action by homeomorphisms that restricted to $G$ is the left multiplication action. If $p \in Y$ is a bounded parabolic point, then $\partial(Stab_{\psi}p) = \{p\}$.
\end{lema}

\begin{proof}Let $U$ be an open neighbourhood of $p$ and $K = G+_{\partial}Y - U$. We have that $K$ is compact. Since $\psi|_{Stab_{\psi}p\times (G+_{\partial}Y-\{p\})}$ is properly discontinuous, the set $\{g \in Stab_{\psi}p: \psi(g,\{1\})\cap K \neq \emptyset \} = \{g \in Stab_{\psi}p: g\in K\}$ is finite, which implies that there exist $g_{U}\in Stab_{\psi}p \cap U$. Thus, $\{g_{U}\}_{U}$ is a net contained in $Stab_{\psi}p$ that converges to $p$, which implies that $p \in \partial(Stab_{\psi}p)$.

Let's suppose that there exists $q \in \partial(Stab_{\psi}p) - \{p\}$. There exists $U$ and $V$ disjoint open sets such that $p \in U$ and $q \in V$. Let $K = G+_{\partial}Y - U$. We have that $K$ is compact and $V \subseteq K$. Since $q$ is in the closure of $Stab_{\psi}p$, there exists a net $\{g_{\gamma}\}_{\gamma \in \Gamma}$ contained in $Stab_{\psi}p$ that converges to $q$. So there exists $\gamma_{0} \in \Gamma$ such that $\forall \gamma > \gamma_{0}$, $g_{\gamma} \in V$, which implies that the set $\{g_{\gamma}\}_{\gamma \in \Gamma}\cap V$ is infinite. But $\{g_{\gamma}\}_{\gamma \in \Gamma}\cap V \subseteq \{g \in Stab_{\psi}p: \psi(g,\{1\})\cap K \neq \emptyset\}$ that must be finite because $p$ is bounded parabolic. Absurd. Thus, $\partial(Stab_{\psi}p) = \{p\}$.
\end{proof}

\begin{prop}\label{quaseconvexidade} Let $G$ be a group, $G +_{\partial} \! Y \! \in EPers(G)$, $L+\varphi: G \curvearrowright G +_{\partial} \! Y$ an action by homeomorphisms, $P \subseteq Y$ the set of bounded parabolic points of $L+\varphi$, $P' \subseteq P$ a representative set of orbits, $\mathcal{C} = \{C_{p}\}_{p\in P'}$ a family of compact Hausdorff spaces and $\{Stab_{\varphi} p+_{\partial_{p}}C_{p}\}_{p\in P'}$ a family of spaces with the equivariant perspective property with actions $\eta = \{\eta_{p}\}_{p\in P'}$, such that $L_{p}+\eta_{p}: Stab_{\varphi} p \curvearrowright Stab_{\varphi} p+_{\partial_{p}}C_{p}$ is by homeomorphisms. If $p \in P$, then $Stab_{\varphi}p$ is dynamically quasiconvex with respect to  $\varphi\ltimes \eta$.
\end{prop}

\begin{proof}Let $(G+_{\partial}Y)\ltimes \mathcal{C} = G+_{\delta}W$, for appropriate choices of $W$ and $\delta$. Let $id+\pi: G+_{\delta}W \rightarrow G+_{\partial}Y$ be the projection map. We have that the relation $\sim = \Delta^{2}X\cup\bigcup\limits_{p\in P} \pi^{-1}(p)^{2}$ is topologically quasiconvex (\textbf{Proposition \ref{quaseconvexidadefibrado}}), which implies that $p \in P$, $\sim_{p} =  \Delta^{2}X\cup\bigcup\limits_{q\in Orb_{\varphi}p} \pi^{-1}(q)^{2}$ is topologically quasiconvex. Since $id+\pi$ is continuous, we have that $\forall q \in P$, $\delta(Stab_{\varphi} q) \subseteq \pi^{-1}(\partial(Stab_{\varphi} q)) = \pi^{-1}(q)$. By \textbf{Proposition \ref{transitivityquasiconvexitytopology}}  $\forall p \in P$,  the relation $\sim'_{p} =  \Delta^{2}X\cup\bigcup\limits_{q\in Orb_{\varphi}p} \delta(Stab_{\varphi}q)^{2}$ is topologically quasiconvex. But, by the construction of  $\varphi\ltimes\eta$, we have that if $g \notin Stab_{\varphi}p$, then $\varphi\ltimes \eta(g,\delta(Stab_{\varphi}p)) = \delta(g Stab_{\varphi}p) = \delta(g (Stab_{\varphi}p)g^{-1}) = \delta(Stab_{\varphi}\varphi(g,p))$  and $\delta(Stab_{\varphi}\varphi(g,p)) = \delta(Stab_{\varphi}p)$ if $g \in Stab_{\varphi} p$ and $\delta(Stab_{\varphi}\varphi(g,p)) \cap \delta(Stab_{\varphi}p) = \emptyset$ otherwise. Thus, by \textbf{Proposition \ref{convdimtop}},  $Stab_{\varphi}p$ is dynamically quasiconvex.
\end{proof}

\subsection{Surjective maps}

\begin{prop}Let $\varphi_{i}: G \curvearrowright X_{i}$, with $i = \{1,2\}$, be properly discontinuous cocompact actions, $X_{i}+_{\partial_{i}}Y \in T_{2}EComp(\varphi_{i})$, and an equivariant continuous map $\pi+id: X_{1}+_{\partial_{1}}Y \rightarrow X_{2}+_{\partial_{2}}Y$. If $X_{2}+_{\partial_{2}}Y \in EPers(\varphi_{2})$, then $X_{1}+_{\partial_{1}}Y \in EPers(\varphi_{1})$.
\end{prop}

\begin{proof}Let $u$ be an element of  $\U_{\partial_{1}}$ and $K$ be a compact subset of $X_{1}$ such that the set $\{g\in G: \varphi_{1}(g,K) \notin Small(u)\}$ is infinite. So there exists a net $\{(\varphi_{1}(g_{i},k_{i,1}),\varphi_{1}(g_{i},k_{i,2})\}_{i\in\Gamma}$, with $\{k_{i,j}\}_{i,j}\subseteq K$ and $(\varphi_{1}(g_{i},k_{i,1}),\varphi_{1}(g_{i},k_{i,2})) \notin u$, that converges to a point $(x,y)\in Y\times Y - \Delta Y$. Since $\pi+id$ is continuous, the net $\{(\pi(\varphi_{1}(g_{i},k_{i,1})),\pi(\varphi_{1}(g_{i},k_{i,2}))\}_{i\in\Gamma} = \{(\varphi_{2}(g_{i},\pi(k_{i,1})),\varphi_{2}(g_{i},\pi(k_{i,2}))\}_{i\in\Gamma}$ converges to $(x,y)$. But $\pi(k_{i,j})\in \pi(K)$ and $\forall v \in \U_{\partial_{2}}$, the set $\{g\in G: \varphi_{2}(g,\pi(K)) \notin Small(v)\}$ is finite. So $\forall v \in \U_{\partial_{2}}$, there exists $i_{0}\in \Gamma:$ $\forall i > i_{0}$, $(\varphi_{2}(g_{i},\pi(k_{i,1}),\varphi_{1}(g_{i},\pi(k_{i,2}))) \in v$, which implies that $x = y$, absurd. Thus $\{g\in G: \varphi_{1}(g,K) \notin Small(u)\}$ is finite.
\end{proof}

\begin{teo}\label{surjectiveness}Let $X,Z$ be Hausdorff compact spaces, $\psi: G \curvearrowright X$ and $\varphi: G \curvearrowright Z$ actions by homeomorphisms, $P$ the set of bounded parabolic points of $Z$ and a continuous surjective equivariant map $\pi: X \rightarrow Z$ such that $\forall x \in Z-P$, $\#\pi^{-1}(x) = 1$. For $p \in P$, take $X = W_{p}+_{\delta_{p}}\pi^{-1}(p)$, where $W_{p} = X - \pi^{-1}(p)$ and $\delta_{p}$ is the only admissible map that gives the original topology. Let also decompose $\psi|_{Stab_{\varphi} p \times X}$ as $\psi_{p}^{1}+\psi_{p}^{2}$, with respect to the decomposition of the space. Those are equivalents:

 \begin{enumerate}

    \item The map $\pi$ is topologically quasiconvex and $\forall p \in P$, $X \in EPers(\psi_{p}^{1})$.

    \item There exists an equivariant homeomorphism $T$ to a parabolic blowup $Z\ltimes \mathcal{C}$ that commutes the diagram:

$$ \xymatrix{ X  \ar[rd]^{\pi} \ar[d]_{T} &  \\
             Z\ltimes \mathcal{C} \ar[r]_{\pi'} & Z} $$

Where $\pi'$ is the projection map.

\end{enumerate}

\end{teo}

\begin{obs}On the case where $X = G+_{\delta}W$, $Z = G+_{\delta}Y$ and $\pi|_{G} = id_{G}$, we have that this hypothesis of perspectivity is stronger then just asking that $\forall p \in P$, $Stab_{\varphi}p+\delta(Stab_{\varphi}p) \in EPers(Stab_{\varphi}p)$.
\end{obs}

\begin{proof}$(\Leftarrow)$ Immediate from the last proposition and \textbf{Corollary \ref{equivalencetopconvexity}}.

$(\Rightarrow)$ Let $P' \subseteq P$ be a representative set of orbits. Let, for $p \in P$, $\sim_{p} = \Delta^{2}(X)\cup \bigcup_{q \neq p}\pi^{-1}(q)$ and $X_{p} = X/ \! \sim_{p}$. Since $\sim_{p}$ is topologically quasiconvex, we have that $X_{p}$ is Hausdorff. If $\pi_{p}: X_{p} \rightarrow Z$ is the quotient map, then $X \cong \lim\{X_{p},\pi_{p}\}_{p\in P}$. We have that $X_{p} = V_{p}+_{\partial_{p}}\pi_{p}^{-1}(p)$, where $V_{p}$ is just $X_{p} - \pi_{p}^{-1}(p)$ and $\partial_{p}$ gives the original topology of $X_{p}$.

Let $\mathcal{C} = \{\pi^{-1}_{p}(p)\}_{p \in P'}$ and $p \in P'$. Since $p$ is a bounded parabolic point, $Stab_{\varphi}p$ acts properly discontinuously and cocompactly on $W_{p}$  and on  $V_{p}$ (\textbf{Proposition \ref{liftingparabolic}}). We have also that the induced action of $Stab_{\varphi} p$ on  $V_{p}+_{\partial_{p}}\pi_{p}^{-1}(p)$  has the perspective property (\textbf{Proposition \ref{quocienteperspectivo}}). Take $\{Stab_{\varphi}p+_{\partial_{p\Lambda}}\pi^{-1}_{p}(p)\}_{p\in P'}$ (where $\Lambda$ is from \textbf{Definition \ref{pielambda}}).  This is a family of spaces with the perspective property and we take $Z\ltimes \mathcal{C}$. We have on this case that $Z\ltimes \mathcal{C} = \lim \{Y_{p},\pi'_{p}\}_{p \in P}$, where $Y_{p} = V_{p}+_{\partial_{p\Lambda\Pi}}\pi_{p}^{-1}(p)$ and $\pi'_{p}: Y_{p} \rightarrow Z$ is the quotient map. But $Y_{p} =  V_{p}+_{\partial_{p\Lambda\Pi}}\pi_{p}^{-1}(p) = V_{p}+_{\partial_{p}}\pi_{p}^{-1}(p) = X_{p}$ (by the perspectivity property) and $\pi'_{p} = \pi_{p}$, which implies that $Z\ltimes \mathcal{C} \cong X$ and the homeomorphism is compatible with $\pi$ and $\pi'$.
\end{proof}

If $\pi$ satisfies the conditions of the theorem, then we call it a blowup map. We use the same terminology of coverings: $Z$ is the base, $X$ is the blowup space and ${\pi^{-1}(p)}_{p}$ are the fibers of the blowup.

\begin{cor}Let $G+_{\delta}W,G+_{\partial}Y\in T_{2}EComp(G)$, with respective actions $L+\psi$ and $L+\varphi$, $P$ the set of bounded parabolic points of $L+\varphi$ and a continuous surjective equivariant map $id+\alpha: G+_{\delta}W \rightarrow G+_{\partial}Y$ that is topologically quasiconvex and $\forall x \in Y-P$, $\#\alpha^{-1}(x) = 1$. For $p \in P$, take $G+_{\delta}W = W_{p}+_{\delta_{p}}\pi^{-1}(p)$, where $W_{p} = G+_{\delta}W - \pi^{-1}(p)$ and $\delta_{p}$ is the only map that gives the original topology. Let also $L_{p}+\psi_{p} = L+\psi$ but decomposed with respect to the new decomposition of the space. If $\forall p \in P$, $G+_{\delta}W \in EPers(L_{p})$, then $\forall p \in P$, $Stab_{\varphi}p$ is dynamically quasiconvex with respect to $L+\psi$. \eod
\end{cor}

\section{Applications}

\subsection{Convergence property}

\begin{teo}\label{semidirectforconv}Let $X,Z$ be Hausdorff compact spaces, $\psi: G \curvearrowright X$ and $\varphi: G \curvearrowright Z$ actions by homeomorphisms, $P$ the set of bounded parabolic points of $Z$ and a continuous surjective equivariant map $\alpha: X \rightarrow Z$ such that $\forall x \in Z-P$, $\#\alpha^{-1}(x) = 1$. If $\forall p \in P$, $\psi|_{Stab_{\varphi}p\times X}$ has the convergence property, then there exists an equivariant homeomorphism $T$ to a parabolic blowup $Z\ltimes \mathcal{C}$ that commutes the diagram:

$$ \xymatrix{ X  \ar[rd]^{\pi} \ar[d]_{T} &  \\
             Z\ltimes \mathcal{C} \ar[r]_{\pi'} & Z} $$

Where $\pi'$ is the projection map.
\end{teo}

\begin{proof}For $p \in P$, take $X = W_{p}+_{\delta_{p}}\pi^{-1}(p)$, where $W_{p} = X - \pi^{-1}(p)$ and $\delta_{p}$ is the only map that gives the original topology. Let also $\psi_{p}^{1}+\psi_{p}^{2} = \psi$ but decomposed with respect to the decomposition of the space. By the \textbf{Theorem \ref{surjectiveness}} it is sufficient to proof that the map $\alpha$ is topologically quasiconvex and $\forall p \in P$, $X \in EPers(\psi_{p}^{1})$.

Let, for $p \in P$, $\sim_{p} = \Delta^{2}X\cup \bigcup_{q \neq p}\alpha^{-1}(q)$ and $X_{p} = X/\sim_{p}$. We have that $X_{p} = V_{p}+_{\partial_{p}}\pi_{p}^{-1}(p)$, where $V_{p}$ is just $X_{p} - \pi_{p}^{-1}(p)$ and $\partial_{p}$ gives the original topology of $X_{p}$. Let, $\forall p \in P$, $\varpi_{p}: X \rightarrow X_{p}$ and $\pi_{p}: X_{p} \rightarrow Z$ be the quotient maps. Since $V_{p}$ is open on $X_{p}$, we have that $\pi_{p}|_{V_{p}}: V_{p} \rightarrow Z-\{p\}$ is a quotient map and since it is injective, we have that $V_{p}$ is homeomorphic to $Z-\{p\}$.

Let $p \in P$. Let $\vartheta_{p}$ be the admissible map such that action of $Stab_{\varphi}p$ on $Stab_{\varphi}p+_{\vartheta_{p}}\pi^{-1}(p)$ has the convergence property. We have that $W_{p}+_{\vartheta_{p\Pi}}\pi^{-1}(p)$ and $V_{p}+_{\vartheta'_{p\Pi}}\pi^{-1}(p)$ have also the convergence property, where $\vartheta_{p\Pi}$ and $\vartheta'_{p\Pi}$ are the admissible maps given by the attractor-sum functors. By the uniqueness of the compactification with the convergence property (Proposition 8.3.1 of \cite{Ge2}) we have that  $W_{p}+_{\vartheta_{p\Pi}}\pi^{-1}(p) = W_{p}+_{\delta_{p}}\pi^{-1}(p)$. Since $\varpi_{p}|_{W_{p}}: W_{p}\rightarrow V_{p}$, $id: \pi^{-1}(p) \rightarrow \pi^{-1}(p)$ and $id: Stab_{\varphi}p+_{\vartheta_{p}}\pi^{-1}(p) \rightarrow Stab_{\varphi}p+_{\vartheta_{p}}\pi^{-1}(p)$ are continuous and $Stab_{\varphi}p$-equivariant, we have, by \textbf{Proposition \ref{functor}}, that the map $\varpi_{p}|_{W_{p}} \! +id: W_{p} \! +_{\vartheta_{p\Pi}}\pi^{-1}(p) \rightarrow V_{p} \! +_{\vartheta'_{p\Pi}}\pi^{-1}(p)$ is continuous. There exists a bijection $f_{p}: X_{p} \rightarrow V_{p}+_{\vartheta'_{p\Pi}}\pi^{-1}(p)$ that commutes the diagram:

$$ \xymatrix{ W_{p} \! +_{\vartheta_{p\Pi}} \! \pi^{-1}(p) \ar[rd]^{\varpi_{p}|_{W_{p}}+id} \ar[d]_{\varpi_{p}} & \\
           X_{p} \ar[r]^<<{ \ \ \ \ \ \ \ \ \ \ f_{p}} & V_{p} \! +_{\vartheta'_{p\Pi}} \! \pi^{-1}(p) } $$

By the universal property of the quotient map, we have that $f_{p}$ is continuous and then a homeomorphism. Thus $X_{p}$ is Hausdorff, which implies that $\alpha$ is topologically quasiconvex (\textbf{Corollary \ref{equivalencetopconvexity}}).

Since the action $\psi_{p}^{1}+\psi_{p}^{2}: Stab_{\varphi}p \curvearrowright W_{p}+_{\vartheta_{p\Pi}}\pi^{-1}(p)$ has the convergence property and $\psi_{p}^{1}$ is properly discontinuous and cocompact, we have that $X \in EPers(\psi_{p}^{1})$ (Proposition 4.15 of \cite{So}).
\end{proof}

On the special case where $\varphi$ is relatively hyperbolic we have:

\begin{prop}Let $X,Z$ be Hausdorff compact spaces, $\psi: G \curvearrowright X$ and $\varphi: G \curvearrowright Z$ actions by homeomorphisms, $P$ the set of bounded parabolic points of $Z$ and a continuous surjective equivariant map $\alpha: X \rightarrow Z$. If  $\psi$ has the convergence property and $\varphi$ is minimal and relatively hyperbolic, then there exists an equivariant homeomorphism $T$ to a parabolic blowup $Z\ltimes \mathcal{C}$ that commutes the diagram:

$$ \xymatrix{ X  \ar[rd]^{\pi} \ar[d]_{T} &  \\
             Z\ltimes \mathcal{C} \ar[r]_{\pi'} & Z} $$

Where $\pi'$ is the projection map.
\end{prop}

\begin{proof}Since $\varphi$ is relatively hyperbolic and minimal, every point of $Z$ is conical or bounded parabolic (Main Theorem of \cite{Ge1}). So the condition that $\forall x \in Z-P$, $\#\alpha^{-1}(x) = 1$ is already satisfied since inverse image of a conic point must have just one point (Proposition 7.5.2 of \cite{Ge2}).
\end{proof}

\begin{cor}Let $X,Z$ be Hausdorff compact spaces, $\psi: G \curvearrowright X$ and $\varphi: G \curvearrowright Z$ actions by homeomorphisms, $P$ the set of bounded parabolic points of $Z$ and a continuous surjective equivariant map $\alpha: X \rightarrow Z$ such that $\forall x \in Z-P$, $\#\alpha^{-1}(x) = 1$. If  $\psi$ has the convergence property, then $\forall p\in P$,  $Stab_{\varphi} p$ is dynamically quasiconvex with respect to $\psi$. \eod
\end{cor} 

\begin{cor}Let $X,Z$ be Hausdorff compact spaces, $\psi: G \curvearrowright X$ and $\varphi: G \curvearrowright Z$ actions by homeomorphisms, $P$ the set of bounded parabolic points of $Z$ and a continuous surjective equivariant map $\alpha: X \rightarrow Z$. If $\psi$ has the convergence property and $\varphi$ is minimal and relatively hyperbolic, then $\forall p\in P$,  $Stab_{\varphi} p$ is dynamically quasiconvex with respect to $\psi$. \eod
\end{cor} 

\subsection{Geometric compactifications}

\begin{lema}\label{noperpofhorospheres}Let $G$ be a group that acts on a metrizable compact space $Z$ with the convergence property and $p \in Z$ a bounded parabolic point. Let $S$ be a finite set of generators of $G$ such that the action on the attractor-sum $G+_{\partial_{c}}Z$ is geometric. Let $d$ be the word metric on $G$ with respect to $S$. Let $\{h_{n}g_{n}\}_{n\in \N} \subseteq G$ be a sequence that converges to $p$ on $G+_{\partial_{c}}Z$ such that $\forall n \in \N$, $h_{n} \in Stab \ p$ and $d(h_{n},h_{n}g_{n}) = d(Stab \ p,h_{n}g_{n})$ and $\forall g \in G$, $(Stab \ p)g \cap \{h_{n}g_{n}\}_{n\in\N}$ is finite. Then the set $\{h_{n}: n \in \N\}$ is infinite.
\end{lema}

\begin{obs}This lemma is essentially proved in Remark 8.3(1) of  \cite{GGPY}.
\end{obs}

\begin{proof}Suppose that the set $\{h_{n}: n \in \N\}$ is finite. Let $h \in Stab \ p$ and $\{hg_{n_{i}}\}_{i\in\N}$ a subsequence of $\{h_{n}g_{n}\}_{n\in\N}$. For every $i \in \N$, take $\gamma_{i}$ a geodesic from $h$ to $hg_{n_{i}}$. Since $\{hg_{n_{i}}\}_{i\in\N}$ converges to $p$, the sequence of geodesics $\{\gamma_{i}\}_{i\in\N}$  has a subsequence $\{\gamma_{i_{j}}\}_{j\in\N}$ that  converges to a geodesic ray $\gamma: [0.\infty] \rightarrow Cay(G,S)+_{\partial'_{c}}Z$ such that $\gamma(0) = h$ and $\gamma(\infty) = p$, where $\partial'_{c}$ is the admissible map of the attractor-sum. This sequence is constructed on the following way: Let $\Phi_{1}$ be a subsequence such that $\Phi_{1}|_{[0,1]}$ converges. If we have $\Phi_{k}$, we define $\Phi_{k+1}$ as a subsequence of $\Phi_{k}$ such that $\Phi_{k+1}|_{[0,k+1]}$ converges. Choose $\Phi = \{\gamma_{i_{j}}\}_{j\in\N}$ such that $\forall j \in \N$, $\gamma_{i_{j}} \in \Phi_{j}$. Since $\forall j \in \N$, eventually $\Phi$ is a subsequence of $\Phi_{j}$, we have that $\forall n\in \N$, $\{\gamma_{i_{j}}|_{[0,n]}\}_{j\in\N}$ converges to a geodesic $\gamma^{n}$. Take $\gamma: [0,\infty] \rightarrow  Cay(G,S)+_{\partial_{c}}Z$ defined by $\gamma(t) = \gamma^{n}(t)$ for some $n$ where $\gamma^{n}(t)$ is defined (it is clear that it agrees if we change $n$) and $\gamma(\infty) = p$. It is clear that $\gamma|_{[0,\infty)}$ is a geodesic ray.

We have that $\forall n \in \N$, there exists only a finite number of points $y$ such that $d(h,y) = d(h,\gamma(n))$. So the sequence $\{\gamma_{i_{j}}(n)\}_{i_{j} \in \N}$ must be eventually constant. Then, $\forall n \in \N$, there exists $j \in \N:$ $\gamma|_{[0,n]} = \gamma_{i_{j}}|_{[0,n]}$. Then, $\forall n \in \N$, $d(\gamma(n),h) = d(\gamma(n),Stab \ p)$, which implies that $\{d(\gamma(n),Stab \ p)\}_{n\in\N}$ is unbounded.

Let $u,u' \in \U_{\partial_{c}}$, with $u'^{2} \subseteq u$, where $\U_{\partial_{c}}$ is the uniform structure compatible with the topology of $G+_{\partial_{c}}Z$. Since this compactification is geometric, there exists $v \in \U_{\partial_{c}}$ such that if a geodesic do not intersect a $v$-small neighbourhood $V$ of $h \in G$, then the geodesic is $u'$-small. Take $t_{0} \in \N$ such that for every $t \geqslant t_{0}$, $\gamma(t) \notin V$. There exists $j_{t} \in \N$ such that $\gamma(t)\in \gamma_{n_{i_{j_{t}}}}([0,d(h,hg_{n_{i_{j_{t}}}})]\cap \N)$ and $(hg_{n_{i_{j_{t}}}},p)\in u'$. We have also that $\gamma_{n_{i_{j_{t}}}}([t_{0},d(h,hg_{n_{i_{j_{t}}}})]\cap \N) \cap V = \emptyset$, which implies that $\gamma_{n_{i_{j_{t}}}}([t_{0},d(h,hg_{n_{i_{j_{t}}}})]\cap \N)$ is $u'$-small. Then $(\gamma(t),hg_{n_{i_{j_{t}}}}) \in u'$, which implies that $(\gamma(t),p) \in u' \subseteq u$. So we have that if $U$ is a $u$-small neighbourhood of $p$, there exists $t_{0} \in \N$ such that $\gamma((t_{0},\infty]\cap \N) \subseteq U$, which implies that $\gamma|_{\N\cap \{\infty\}}$ is continuous in $\infty$ and then it is continuous.

We have that $Stab \ p \cup \{p\}$ is closed and the action of $G$ on $Z$ is geometric, which implies that $W$, the union of the images of geodesic, geodesic rays and geodesic lines linking points in $Stab \ p \cup \{p\}$, is closed in $G+_{\partial_{c}}Z$ (Main Lemma of \cite{GP}).

Let $\varsigma: G+_{\partial_{c}}Z - \{p\} \rightarrow (G+_{\partial_{c}}Z - \{p\})/ Stab \ p$ be the projection map. We have that $W - \{p\} \subseteq G$ is closed in $G+_{\partial_{c}}Z - \{p\}$ and $\varsigma$-saturated. We also have that $\forall w \in W$, $Hw$ is open in $G+_{\partial_{c}}Z - \{p\}$. So $\varsigma(W - \{p\})$ is closed and discrete on the compact space $(G+_{\partial_{c}}Z - \{p\})/ Stab \ p$, which implies that $\varsigma(W - \{p\})$ is finite. Then, there exists $S$ a finite subset of $G$ such that $W - \{p\} \subseteq HS$. But $\{d(w,Stab \ p): w \in W\}$ is bounded, contradicting the fact that the image of $\gamma$ is contained in $W$.

Thus the set $\{h_{n}: n \in \N\}$ is infinite.
\end{proof}

\begin{obs}The Main Lemma of \cite{GP} is stated only for relatively hyperbolic actions. However, the same proof works for geometric compactifications if we ask about geodesic hull instead of quasi-geodesic hull (we needed just the geodesic hull on the proof of our lemma).
\end{obs}

\begin{lema}Let $G+_{\partial}Y$ be a compact metrizable space and consider an action by homeomorphisms $L+\varphi: G\curvearrowright G +_{\partial} Y$. Let $p \in Y$ be a point such that $\varphi|_{Stab_{\varphi}p\times Y-\{p\}}$ is cocompact. Then $L+\varphi|_{Stab_{\varphi}p\times G+_{\partial}Y-\{p\}}$ is cocompact. if and only if there is no infinite set $A \subseteq \{Stab_{\varphi}(p)g: g \in G\}$ such that  $\partial(\bigcup A) = \{p\}$.
\end{lema}

\begin{proof}$(\Rightarrow)$ Let $K$ be a compact subset of $G+_{\partial}Y \! -\{p\}$ such that $\varphi(Stab_{\varphi}p,K) = G+_{\partial}Y-\{p\}$. So $\forall g \in G$, $\exists h_{g} \in K\cap Stab_{\varphi}(p)g$. Suppose that there is an infinite set $A \subseteq \{Stab_{\varphi}(p)g: g \in G\}$ such that $\partial(\bigcup A) = \{p\}$. So there is a sequence $\{g_{n}\}_{n\in\N} \subseteq \bigcup A$ such that $\forall n \neq m$, $Stab_{\varphi}(p)g_{m} \neq Stab_{\varphi}(p)g_{m}$. We have that $\{h_{g_{n}}\}_{n\in\N} \subseteq \bigcup A$, which implies that $\{h_{g_{n}}\}_{n\in\N}$ converges to $p$, contradicting the fact that $K$ is compact and $p \notin K$. Thus, there is no infinite set  $A \subseteq \{Stab_{\varphi}(p)g: g \in G\}$ such that  $\partial(\bigcup A) = \{p\}$.

$(\Leftarrow)$ Let $p$ be a  point such that $\varphi|_{Stab_{\varphi}p\times Y-\{p\}}$ is cocompact. Then there exists $K \subseteq Y-\{p\}$ a compact set such that $\varphi(Stab_{\varphi}p,K) = Y-\{p\}$. Let $U$ be an open neighbourhood of $K$ in $G+_{\partial}Y$ such that $p\notin Cl_{G+_{\partial}Y}(U)$. Let $J = L+\varphi(Stab_{\varphi}p,Cl_{G+_{\partial}Y}(U))$. We have that $J$ is open and $J\cap Y = Y-\{p\}$.

Let $A = \{Stab_{\varphi}(p)g: g \in G, Stab_{\varphi}(p)g\cap J =\emptyset\}$. Let $\{g_{n}\}_{n\in \N} \subseteq \bigcup A$ be a wandering sequence. Since $\{g_{n}\}_{n\in \N}\cap J = \emptyset$, and $J$ is open, there is no cluster point of $\{g_{n}\}_{n\in \N}$ in $Y-\{p\}$, which implies that $\{g_{n}\}_{n\in \N}$ converges to $p$. So $\partial(\bigcup A) = \{p\}$. Then, by hypothesis, $A$ is finite.

Take $a_{1},...,a_{n} \in G$ such that  $A = \{Stab_{\varphi}(p)a_{1},...,Stab_{\varphi}(p)a_{n}\}$. Then $K' =  Cl_{G+_{\partial}Y}(U)\cup \{a_{1},...,a_{n}\}$ is a compact set such that $L+\varphi(Stab_{\varphi}p,K') = G+_{\partial}Y-\{p\}$.
\end{proof}

\begin{lema}Let $L+\varphi: G\curvearrowright G+_{\partial}Y$ be a geometric action and $p \in Y$ a bounded parabolic point of $\varphi$. If $A \subseteq \{Stab_{\varphi}(p)g: g \in G\}$ such that $A$ and $\partial(\bigcup A) = \{p\}$, then $A$ is finite.
\end{lema}

\begin{proof}Suppose there exists $A \subseteq \{Stab_{\varphi}(p)g: g \in G\}$ such that $A$ is infinite and $\partial(\bigcup A) = \{p\}$. Take $\{g_{n}\}_{n\in\N}$ a wandering sequence such that $\forall n \in \N$, $Stab_{\varphi}(p)g_{n} \in A$ and $d(1,g_{n}) = d(Stab_{\varphi}(p),g_{n})$. Since $\partial(\bigcup A) = \{p\}$, we have that $\{g_{n}\}_{n\in\N}$ converges to $p$, a contradiction to \textbf{Lemma \ref{noperpofhorospheres}}.
\end{proof}

\begin{prop}\label{parabolicpoints}Let $L+\varphi: G\curvearrowright G+_{\partial}Y$ be a geometric action and $p \in Y$. Then $p$ is bounded parabolic with respect to $L+\varphi$ if and only if it is bounded parabolic with respect to $\varphi$.
\end{prop}

\begin{proof}$(\Rightarrow)$ Immediate.

$(\Leftarrow)$ Let $p$ be a bounded parabolic point with respect to $\varphi$. We have that $Stab_{\varphi}p$ acts with the convergence property on $G+_{\partial}Y$, which implies that it acts properly discontinuously outside the limit set of  $Stab_{\varphi}p$ (\cite{Bo2}). But the limit set is just $\{p\}$, which implies that  $Stab_{\varphi}p$ acts properly discontinuously on $G+_{\partial}Y - \{p\}$. The cocompactness comes from the previous two lemmas.
\end{proof}

\begin{obs}For relatively hyperbolic actions this proposition is a simple consequence of the Main Theorem (b) of \cite{Ge1}.
\end{obs}

\begin{teo}\label{geoconv}Let $G$ be a finitely generated group, $L+\varphi: G \curvearrowright G+_{\partial_{c}}Z$ a geometric action  $P \subseteq Z$ the set of bounded parabolic points, $P' \subseteq P$ a representative set of orbits, $\mathcal{C} = \{C_{p}\}_{p\in P'}$ a family of compact Hausdorff spaces and $\{Stab_{\varphi} p+_{\partial_{p}}C_{p}\}_{p\in P'}$ a family of spaces with convergence actions $\eta = \{\eta_{p}\}_{p\in P'}$, with $L_{p}+\eta_{p}: Stab_{\varphi} p \curvearrowright Stab_{\varphi} p +_{\partial_{p}}C_{p}$ (where $L_{p}$ is the left multiplication action). Then the action $\varphi \ltimes \eta: G \curvearrowright X$ has the convergence property, where $X = Z \ltimes \mathcal{C}$.
\end{teo}

\begin{obs}During this proof, when there is any subset of $G$, we change freely between the elements of the group and their respective actions on the spaces $X$, $X_{p}$ and $Z$. In each sentence it is clear which meaning we are using. So there should be no ambiguity.
\end{obs}

\begin{proof}Let $\pi: X \rightarrow Z$, $\varpi_{p}: X \rightarrow X_{p}$ and $\pi_{p}: X_{p} \rightarrow Z$ be the projection maps, where $X_{p} = Z-\{p\}+_{\partial_{p\Pi}}C_{p}$. Let $\Phi \subseteq G$ be a wandering sequence.

Let's suppose that there exists $p \in P$ such that $\Phi\cap Stab_{\varphi}p$ is infinite. Since the action of $Stab_{\varphi}p$ on $X_{p}$ has convergence property and $\varpi_{p}$ is $Stab_{\varphi}p$-equivariant and non ramified over the limit set of $X_{p}$, we have that $\varphi \ltimes \eta$ restricted to $Stab_{\varphi}p$ has the convergence property (Proposition 7.6.1 of \cite{Ge2}). So there exists a collapsing subsequence $\Phi'$ of $\Phi$.

Let's suppose now that $\forall p \in P$, $\Phi \cap Stab_{\varphi}p$ is finite.

Suppose that there exists $g \in G$ and $p \in P$ such that $\Phi \cap (Stab_{\varphi}p)g$ is infinite. Then $\Phi g^{-1} \cap (Stab_{\varphi}p)$ is infinite, where $\Phi g^{-1} = \{g_{n}g^{-1}\}_{n\in\N}$ and $\Phi = \{g_{n}\}_{n\in\N}$. We have that there exists $\Phi'$ a collapsing subsequence of $\Phi g^{-1}$ with attracting point $r \in X$ and repelling point $s \in X$. So $\Phi'g$ is a collapsing subsequence of $\Phi$ with attracting point $r$ and repelling point $\varphi \ltimes \eta(g^{-1},s)$.

Let's suppose now that $\forall g \in G$, $\forall p\in P$, $\Phi \cap (Stab_{\varphi}p)g$ is finite.

Since $\varphi$ has the convergence property, there exists $\Phi'$ a collapsing subsequence of $\Phi$ with attracting point $a \in Z$ and repelling point $b \in Z$.

Suppose that $a$ and $b$ are not bounded parabolic. Then $\#\pi^{-1}(a) = \#\pi^{-1}(b) = 1$. We have that $\Phi'|_{Z-\{b\}}$ converges uniformly to $a$, which implies that $\Phi'|_{X - \{\pi^{-1}(b)\}}$ converges uniformly to $\pi^{-1}(a)$ (\textbf{Proposition \ref{liftingnet}}).

Suppose that $a$ or $b$ is bounded parabolic.

If $a$ is not bounded parabolic, then $b$ is bounded parabolic, which implies that it is not a conical point. So $ \exists b' \neq b$ such that $Cl_{Z^{2}}(\{(\varphi(g,b), \varphi(g,b')): g \in \Phi'\}) \cap \Delta Z \neq \emptyset$. Since the sequence defined by $\Phi'|_{b'}$ converges to $a$, we have that there exists $\Phi''$ a subsequence of $\Phi'$ such that $\Phi''|_{b}$ converges to $a$. So $\Phi''|_{Z}$ converges uniformly to $a$, which implies that $\Phi''|_{X}$ converges uniformly to $\pi^{-1}(a)$ (since $\#\pi^{-1}(a) = 1$).

If $a$ is bounded parabolic, take $\Phi''$ a subsequence of $\Phi'$ such that $\forall g \in G$, $\#\Phi'' \cap (Stab_{\varphi}a)g \leqslant 1$. By the \textbf{Lemma \ref{noperpofhorospheres}}, we can write $\Phi''$ as $\{h_{n}g_{n}\}_{n\in\N}$ with $d(h_{n},h_{n}g_{n}) = d(Stab_{\varphi}a,h_{n}g_{n})$ (where $d$ is the word metric with respect to a finite set of generators of $G$ such that $L +\varphi$ is geometric) and $\{h_{n}\}_{n \in \N}$ is a wandering sequence in $Stab_{\varphi}a$. We also have that $\{g_{n}\}_{n\in\N}$ is a wandering sequence. Since  $d(1,g_{n}) = d(Stab_{\varphi}a,g_{n})$, we have by the \textbf{Lemma \ref{noperpofhorospheres}} that $a$ is not a cluster point of $\{g_{n}\}_{n\in\N}$. Let $\Phi''' = \{h_{n_{i}}g_{n_{i}}\}_{i\in\N}$ be a subsequence of $\Phi''$ such that $\{g_{n_{i}}\}_{i\in\N}$ is a collapsing subsequence of $\{g_{n}\}_{n\in\N}$ with attracting point $a' \in Z$ and repelling point $b' \in Z$ and $\{h_{n_{i}}\}_{i\in\N}$ is a collapsing subsequence of $\{h_{n}\}_{n\in\N}$ with attracting point $c \in C_{a}$ and repelling point $d \in C_{a}$. We have that $a'$ is a cluster point of $\{g_{n}\}_{n\in\N}$, which implies that $a \neq a'$.

We have that $\{g_{n_{i}}\}_{i \in \N}|_{Z-\{b'\}}$ converges uniformly to $a'$, which implies that the set of liftings of $\{g_{n_{i}}\}_{i \in \N}|_{Z-\{b'\}}$ to $X_{a}$ converges uniformly to $\pi_{a}^{-1}(a')$. We have also that $\{h_{n_{i}}\}_{i\in\N}|_{X_{a} - \{d\}}$ converges uniformly to $c$. Since $\pi_{a}^{-1}(a')\neq d$, the set of liftings of $\Phi'''|_{Z-\{b'\}}$ to $X_{a}$ converges uniformly to $c$ (\textbf{Proposition \ref{composicaouniforme}}). Then $\Phi'''|_{X-\{\pi^{-1}(b')\}}$, which is contained on the set of liftings of $\Phi'''|_{Z-\{b'\}}$ to $X$, converges uniformly to $\varpi_{a}^{-1}(c)$.

If $b'$ is a conical point, then $\#\pi^{-1}(b') = 1$ and we are done. If $b'$ is not a conical point, there exists $b'' \neq b'$ such that $Cl_{Z^{2}}\{(\varphi(g_{n_{i}},b'),\varphi(g_{n_{i}},b'')): i \in \N\} \cap \Delta Z \neq \emptyset$. Since $b'' \neq b'$, we have that $\{\varphi(g_{n_{i}},b'')\}_{i \in \N}$ converges to $a'$. So $a'$ is a cluster point of $\{\varphi(g_{n_{i}},b')\}_{i \in \N}$. Let $\Phi^{(4)} = \{h_{n_{i_{j}}}g_{n_{i_{j}}}\}_{j \in \N}$ be a subsequence of $\Phi'''$ such that $\{\varphi(g_{n_{i_{j}}},b')\}_{j \in \N}$ converges to $a'$. So the set of liftings of $\{g_{n_{i_{j}}}\}_{j\in\N}|_{\{b'\}}$ to $X_{a}$ converges uniformly to $\pi_{a}^{-1}(a')$, then the set of liftings of $\Phi^{(4)}|_{\{b'\}}$ to $X_{a}$ converges uniformly to $c$ and then the set of liftings of $\Phi^{(4)}|_{\{b'\}}$ to $X$ (that contains $\Phi^{(4)}|_{\{\pi^{-1}(b')\}}$) converges uniformly to $\varpi_{a}^{-1}(c)$. Since $\Phi^{(4)}|_{X-\{\pi^{-1}(b')\}}$ already converges uniformly to $\varpi_{a}^{-1}(c)$, then $\Phi^{(4)}$ converges uniformly on $X$ to $\varpi_{a}^{-1}(c)$.

Thus $\varphi \times \eta$ has the converges property.
\end{proof}

\begin{cor}Let $X,Z$ be Hausdorff compact spaces, $G$ a finitely generated group, $\psi: G \curvearrowright X$ and $\varphi: G \curvearrowright Z$ actions by homeomorphisms, $P$ the set of bounded parabolic points of $Z$ and a continuous surjective equivariant map $\alpha: X \rightarrow Z$ such that $\forall x \in Z-P$, $\#\alpha^{-1}(x) = 1$. If $\forall p \in P$, $\psi|_{Stab_{\varphi}p\times X}$ has the convergence property, $\varphi$ has the convergence property and $G+_{\partial_{c}}Z$ is geometric with respect to some finite set of generators of $G$, then $\psi$ has the convergence property.
\end{cor}

\begin{proof}Immediate from the last theorem and \textbf{Theorem \ref{semidirectforconv}}.
\end{proof}

When the base of a parabolic blowup comes from a geometric compactification, we have a characterization of the blowup maps that is stronger than \textbf{Theorem \ref{surjectiveness}}:

\begin{teo}\label{geoblowup}Let $X,Z$ be metrizable compact spaces, $\psi: G \curvearrowright X$ an action by homeomorphisms, $L+\varphi: G \curvearrowright  G+_{\partial_{c}}Z$ a geometric action, $P$ the set of bounded parabolic points of $Z$ and a continuous surjective equivariant map $\pi: X \rightarrow Z$ such that $\forall x \in Z-P$, $\#\pi^{-1}(x) = 1$. For $p \in P$, take $X = W_{p}+_{\delta_{p}}\pi^{-1}(p)$, where $W_{p} = X - \pi^{-1}(p)$ and $\delta_{p}$ is the only admissible map that gives the original topology. Let also decompose $\psi|_{Stab_{\varphi} p \times X}$ as $\psi_{p}^{1}+\psi_{p}^{2}$, with respect to the decomposition of the space. Then $\pi$ is a blowup map if and only if $\forall p \in P$, $X \in EPers(\psi_{p}^{1})$.
\end{teo}

\begin{obs}This proof is surprisingly analogous to Theorem \ref{geoconv}.
\end{obs}

\begin{proof} Let $\U$ be the uniform structure compatible with the topology of $X$ and $u \in \U$. Let $p \in P$, $\Phi = \{g\in G: \psi(g, \pi^{-1}(p)) \notin Small(u)\}$ and $\tilde{\Phi} = \{gStab_{\varphi} p: g \in \Phi\}$. Suppose that $\tilde{\Phi}$ is infinite. Let $\Phi' \subseteq \Phi$ be a set of representatives of the cosets in $\tilde{\Phi}$. We have that $\Phi'$ is also infinite. 

Let $q \in P - \{p\}$. Since the action of $Stab_{\varphi}q$ on $X$ is an element of $EPers(\psi_{q}^{1})$, we have that $\{g\in  Stab_{\varphi} q: \psi(g, \pi^{-1}(p)) \notin Small(u)\}$ is finite. So $\forall q \in P$, $\Phi' \cap Stab_{\varphi}q$ is finite.

Let  $q \in P$ and $g \in G$. Suppose that  $\psi(g,p) \neq q$. We have that  $\Phi' g^{-1} = \{hg^{-1}\in G: \psi(h, \pi^{-1}(p)) \notin Small(u)\}  = \{h\in G: \psi(h, \psi(g,\pi^{-1}(p))) \notin Small(u)\} = \{h\in G: \psi(h, \pi^{-1}(\varphi(g,p))) \notin Small(u)\}$. By the last paragraph we have that $\Phi' g^{-1} \cap Stab_{\varphi}q$ is finite. So $\Phi' \cap (Stab_{\varphi}q) g$ is finite. Suppose now that  $\psi(g,p) = q$. If there is another $g' \in G$ such that  $\psi(g',p) = q$, then $g'^{-1}g \in Stab_{\varphi}p$, which implies that $g Stab_{\varphi}p = g' Stab_{\varphi}p$. But every element $h \in (Stab_{\varphi}q) g$ satisfies $\psi(h,p) = q$.  Since $\Phi'$ is a set of representatives of the left cosets that are in $\tilde{\Phi}$, then $\# \Phi' \cap (Stab_{\varphi}q) g \leqslant 1$. So $\forall q \in P$, $\forall g \in G$, $\Phi' \cap (Stab_{\varphi}q) g$ is finite.

Since $\varphi$ has the convergence property, there exists $\Phi'' \subseteq \Phi'$ a collapsing sequence with attracting point $a \in Z$ and repelling point $b \in Z$.

Suppose that $a$ is not bounded parabolic and $b \neq p$. Then $\#\pi^{-1}(a) = 1$. We have that $\Phi''|_{Z-\{b\}}$ converges uniformly to $a$, which implies that $\Phi''|_{X - \{\pi^{-1}(b)\}}$ converges uniformly to $\pi^{-1}(a)$ (\textbf{Proposition \ref{liftingnet}}). Since $b\neq p$, $\pi^{-1}(p)\subseteq X - \{\pi^{-1}(b)\}$, which implies that there is $g \in \Phi''$ such that $\psi(g,\pi^{-1}(p))$ is contained on a $v$-neighbourhood of $\pi^{-1}(a)$, where $v \in \U$ is such that $v^{2} \subseteq u$, contradicting the fact that $\psi(g,\pi^{-1}(p))$ is not $u$-small (since $g\in \Phi$).

So $a$ is bounded parabolic or $b = p$.

If $a$ is not bounded parabolic, then $b = p$, which implies that $b$ is not a conic point. So $ \exists b' \neq b$ such that $Cl_{Z^{2}}(\{(\varphi(g,b), \varphi(g,b')): g \in \Phi'\}) \cap \Delta Z \neq \emptyset$. Since the sequence defined by $\Phi''|_{b'}$ converges to $a$, we have that there exists $\Phi'''$ a subsequence of $\Phi''$ such that $\Phi'''|_{b}$ converges to $a$. So $\Phi'''|_{Z}$ converges uniformly to $a$, which implies that $\Phi'''|_{X}$ converges uniformly to $\pi^{-1}(a)$ (since $\#\pi^{-1}(a) = 1$), a contradiction analogous to the previous one.

So $a$ is bounded parabolic. Take $\Phi'''$ a subsequence of $\Phi''$ such that $\forall g \in G$, $\#\Phi''' \cap (Stab_{\varphi}a)g \leqslant 1$. By the \textbf{Lemma \ref{noperpofhorospheres}}, we can write $\Phi'''$ as $\{h_{n}g_{n}\}_{n\in\N}$ with $d(h_{n},h_{n}g_{n}) = d(Stab_{\varphi}a,h_{n}g_{n})$ (where $d$ is the word metric with respect to a finite set of generators of $G$ such that $L +\varphi$ is geometric) and $\{h_{n}\}_{n \in \N}$ is a wandering sequence in $Stab_{\varphi}a$. We also have that $\{g_{n}\}_{n\in\N}$ is a wandering sequence. Since  $d(1,g_{n}) = d(Stab_{\varphi}a,g_{n})$, we have by the \textbf{Lemma \ref{noperpofhorospheres}} that $a$ is not a cluster point of $\{g_{n}\}_{n\in\N}$. Let $\Phi^{(4)}= \{h_{n_{i}}g_{n_{i}}\}_{i\in\N}$ be a subsequence of $\Phi'''$ such that $\{g_{n_{i}}\}_{i\in\N}$ is a collapsing subsequence of $\{g_{n}\}_{n\in\N}$ with attracting point $a' \in Z$ and repelling point $b' \in Z$. We have that $a'$ is a cluster point of $\{g_{n}\}_{n\in\N}$, which implies that $a \neq a'$.

We have that $\{g_{n_{i}}\}_{i \in \N}|_{Z-\{b'\}}$ converges uniformly to $a'$. Suppose that $b' \neq p$. Then the set $K = \{\varphi(g_{n_{i}},p): i\in \N\} \cup \{a'\}$ is compact and, by taking a subsequence, we can suppose that $a \notin  K$ (since $a \neq a'$). Then $\pi^{-1}(K)$ is a compact set that does not intersect $\pi^{-1}(a)$. Since the action of $Stab_{\varphi}a$ on $X$ is an element of $EPers(\psi_{a}^{1})$, the set  $\{h_{n_{i}}: i \in \N,$ $\psi(h_{n_{i}},\pi^{-1}(K)) \notin Small(u) \}$ is finite, which implies that there is $j \in \N$ such that $\psi(h_{n_{j}},\pi^{-1}(K)) \in Small(u)$ . But $\varphi(g_{n_{j}},p) \in K$, which implies that  $\psi(h_{n_{j}}g_{n_{j}},\pi^{-1}(p)) = \psi(h_{n_{j}},\pi^{-1}(\varphi(g_{n_{j}},p))) \subseteq \psi(h_{n_{j}},\pi^{-1}(K))$  and then $\psi(h_{n_{j}}g_{n_{j}},\pi^{-1}(p)) \in Small(u)$, contradicting the fact that $h_{n_{j}}g_{n_{j}} \in \Phi$. 

Then $b' = p$, which implies that $b'$ is not conical. Then there exists $b'' \neq b'$ such that $Cl_{Z^{2}}\{(\varphi(g_{n_{i}},b'),\varphi(g_{n_{i}},b'')): i \in \N\} \cap \Delta Z \neq \emptyset$. Since $b'' \neq b'$, we have that $\{\varphi(g_{n_{i}},b'')\}_{i \in \N}$ converges to $a'$. So $a'$ is a cluster point of $\{\varphi(g_{n_{i}},b')\}_{i \in \N}$. Let $\Phi^{(5)} = \{h_{n_{i_{j}}}g_{n_{i_{j}}}\}_{j \in \N}$ be a subsequence of $\Phi^{(4)}$ such that $\{\varphi(g_{n_{i_{j}}},b')\}_{j \in \N}$ converges to $a'$. So take the compact set $K = \{\varphi(g_{n_{i_{j}}},b'): j\in \N\} \cup \{a'\}$  and get a contradiction analogous to the previous one.

Thus $\tilde{\Phi}$ is finite. This means that $\forall p \in P$, $\pi^{-1}(p)$ is $\psi$-quasiconvex. Since $\forall x \in X-P$, $\#\pi^{-1}(x) = 1$, we have that $\pi$ is topologically quasiconvex (\textbf{Propositioon \ref{convdimtopsub}}). Then, by \textbf{Theorem \ref{surjectiveness}}, $\pi$ is a blowup map.
\end{proof}

\subsection{Relatively hyperbolic actions}

Let $G$ be a finitely generated group, $\varphi: G \curvearrowright Z$ an action by homeomorphisms, $P \subseteq Z$ be the set of bounded parabolic points, $P' \subseteq P$ a representative set of orbits, $\mathcal{C} = \{C_{p}\}_{p\in P'}$ a family of compact Hausdorff spaces, $H = \{H_{p}\}_{p\in P'}$, with $H_{p} \subseteq G$ minimal sets such that $1\in H_{p}$ and $Orb_{\varphi|_{H_{p}\times Z}} p = Orb_{\varphi} p$ and $\{Stab_{\varphi} p+_{\partial_{p}}C_{p}\}_{p\in P'}$ a family of spaces with convergence actions $\eta = \{\eta_{p}\}_{p\in P'}$, with $L_{p}+\eta_{p}: Stab_{\varphi} p \curvearrowright Stab_{\varphi} p +_{\partial_{p}}C_{p}$ (where $L_{p}$ is the left multiplication action).

Let $X = Z\ltimes \mathcal{C}$. Let $\pi: X \rightarrow Z$, and, for $p \in P$, $\varpi_{p}: X \rightarrow X_{p}$ and $\pi_{p}: X_{p} \rightarrow Z$ be the projection maps.

Let, for $p \in P$, $B_{p} \subseteq C_{p}$ be the set of bounded parabolic points with respect to the action $\eta_{p}$.

\begin{prop}\label{subindoparabolicos}$\bigcup\limits_{p \in P}\varpi_{p}^{-1}(B_{p})$ is the set of bounded parabolic points with respect to the action $\varphi\ltimes \eta$.
\end{prop}

\begin{proof}Let $x \in B_{p}$. We have that $\#\varpi_{p}^{-1}(x) = 1$ and, by \textbf{Proposition \ref{liftingparabolic2}}, $\varpi_{p}^{-1}(x)$ is bounded parabolic. So the points of the set $\bigcup\limits_{p \in P}\pi_{p}^{-1}(B_{p})$ are bounded parabolic.

Let $x \in X$ be a bounded parabolic point. If $x \notin P$, then $x = \pi^{-1}(\pi(x))$, which implies that $\pi(x)$ is bounded parabolic, absurd. So $\pi(x) = p \in P$. We have that $\#\varpi_{p}^{-1}(\varpi_{p}(x)) = 1$, which implies that $\varpi_{p}(x)$ is bounded parabolic and then $x \in \varpi_{p}^{-1}(B_{p})$.

Thus, the set of bounded parabolic points of $\varphi\ltimes \eta$ is given by the set $\bigcup\limits_{p \in P}\varpi_{p}^{-1}(B_{p})$. \end{proof}

Suppose now that $\varphi$ has the convergence property and the attractor-sum space $G+_{\partial_{c}}Z$ is geometric. From \textbf{Theorem \ref{geoconv}}, we have that $\varphi \times \eta$ has the convergence property.

Let $A_{p} \subseteq C_{p}$ and $A \subseteq Z$ be the sets of conical points with respect to $\eta_{p}$ and $\varphi$, respectively.

\begin{prop}The points of $\pi^{-1}(A) \cup \bigcup\limits_{p \in P}\varpi_{p}^{-1}(A_{p})$ are conical.
\end{prop}

\begin{proof}Since $\pi$ is a continuous $G$-equivariant map, we have that the points of $\pi^{-1}(A)$ are conical. We have also that $\forall p \in P$, $\varpi_{p}$ is $Stab_{\varphi}p$-equivariant, which implies that the points of $\pi_{p}^{-1}(A_{p})$ are conical with respect to $Stab_{\varphi}p$ and then conical with respect to $G$. Thus the points of $\pi^{-1}(A) \cup \bigcup\limits_{p \in P}\varpi_{p}^{-1}(A_{p})$ are conical.
\end{proof}

\begin{cor}If $\forall p \in P'$, $A_{p} = C_{p}$ and $A = Z - P$, then every point of $X$ is conical. \eod
\end{cor}

\begin{cor}\label{hiperb}If $G$ is finitely generated and $\varphi$ is minimal, then the functor $\varphi \ltimes$ sends a set of minimal hyperbolic actions to a minimal hyperbolic action. In this case, we have that $\partial_{\infty} G = Z \ltimes \{\partial_{\infty}Stab_{\varphi}p\}_{p\in P'}$. \eod
\end{cor}

\begin{cor}If $\forall p \in P'$, $A_{p} \cup B_{p}= C_{p}$ and $A \cup P = Z$, then every point of $X$ is conical or bounded parabolic. \eod
\end{cor}

\begin{cor}\label{relhiperb}If $G$ is finitely generated, and $\varphi$ is a minimal relatively hyperbolic action, then the functor $\varphi \ltimes$ sends sets of relatively hyperbolic actions $\eta = \{\eta_{p}\}_{p \in P'}$ to relatively hyperbolic actions $\varphi\ltimes \eta$. In this case, we have that $\partial_{B} (G,\bigcup_{p\in P'}\mathcal{P}_{p}) = Z \ltimes \{\partial_{B}(Stab_{\varphi}p,\mathcal{P}_{p})\}_{p\in P'}$, where $\mathcal{P}_{p}$ is a set of representatives of maximal parabolic subgroups of the action $\eta_{p}$. \eod
\end{cor}

On another words, we have a proof of a result from Drutu and Sapir \cite{DS}:

\begin{cor}(Drutu and Sapir, Corollary 1.14 of \cite{DS}) If $G$ is finitely generated and relatively hyperbolic with respect to the set $\mathcal{P}$ of representatives of maximal parabolic subgroups and each $P \in \mathcal{P}$ is  relatively hyperbolic with respect to the set of representatives $\mathcal{P}_{P}$, then $G$ is relatively hyperbolic with respect to the set of representatives $\bigcup\mathcal{P}_{P}$.
\end{cor}

Furthermore, we have a topological description of $\partial_{B} (G,\bigcup\mathcal{P}_{P})$ as a limit: it is the space $\partial_{B}(G,\mathcal{P}) \ltimes \{\partial_{B}(P,\mathcal{P}_{P})\}_{P\in \mathcal{P}}$.

\begin{prop}Let $G$ be a finitely generated group and $Z$ a Hausdorff compact space. Suppose that there is a relatively hyperbolic action $\varphi: G \curvearrowright Z$ that is maximal over the relatively hyperbolic actions (i.e. if there exists another relatively hyperbolic action on $Z'$, then there is an equivariant continuous map from $Z$ to $Z'$). Then, for every bounded parabolic point $p \in Z$, $Stab_{\varphi}p$ has no non-trivial relatively hyperbolic action. \eod
\end{prop}

\begin{obs}This happens trivially when $G$ is hyperbolic and when $G$ is a Kleinian group with abelian maximal parabolic subgroups with rank bigger than 1.
\end{obs}

\subsection{Blowing up the Cantor set}

On this subsection let's consider $K$ as the Cantor set again.

\begin{prop}\label{caracterizacaotopologicapbucantor}Let $G$ be a group that acts by homeomorphisms on a metrizable space $X$ and on $K$ and let $\pi: X \rightarrow K$ be a continuous equivariant map. If the action on $K$ has countably many bounded parabolic points, each bounded parabolic point has a dense orbit on $K$ and $\pi$ is a blowup map, then $X$ is a dense amalgam with $\mathcal{Y}$ the set of preimages of bounded parabolic points. \eod
\end{prop}

\begin{obs}As a special case, the theorem holds if $G$ is finitely generated, the action on $K$ is a minimal convergence action and $\pi$ is a blowup map.
\end{obs}

Let $\mathcal{G}$ be a finite graph of groups with finitely generated vertex groups $\{G_{i}\}_{i\in V}$, finite edge groups and fundamental group $G$. From Dahmani Combination Theorem \cite{Da2} $G$ is relatively hyperbolic relative to $\{G_{i}\}_{i\in\Gamma}$ and its Bowditch boundary $\partial_{B}(G,\{G_{i}\}_{i \in V})$ is homeomorphic to $K$. Let $\mathcal{C} = \{C_{i}\}_{i\in V}$ be a family of compact metrizable spaces and, for $i \in V$, $G_{i}+_{\partial_{i}}C_{i}$ a perspective compactification of $G_{i}$ (we suppose that $\partial_{i}(G_{i}) = C_{i}$). Then, by the last proposition, the parabolic blowup $\partial_{B}(G,\{G_{i}\}_{i \in V})\ltimes \mathcal{C}$ is a dense amalgam with $\mathcal{Y}$ the set of preimages of bounded parabolic points.

\begin{prop}\label{combinationconvergence}(Combination Theorem for Convergence Actions) Let $\mathcal{G}$ be a finite graph of groups with finitely generated vertex groups $\{G_{i}\}_{i\in V}$, finite edge groups and fundamental group $G$. Suppose that for each $i \in V$, $G_{i}$ acts minimally and with convergence property on a space $C_{i}$. Then $G$ acts with convergence property on a space $X$ such that $\forall i \in V$, the limit set $\Lambda \iota_{i}(G_{i})$ is equivariatly homeomorphic to $C_{i}$, where $\iota_{i}: G_{i} \rightarrow G$ is the canonical inclusion map and $\Lambda H$ is the set of limit points of $H$. Moreover, we can construct $X$ as a dense amalgam with $\mathcal{Y} = \{g\Lambda \iota_{i}(G_{i}): g \in G, i \in V\}$. \eod
\end{prop}

If, for each $i \in V$, $G_{i}$ is relatively hyperbolic with respect to $\mathcal{P}_{i}$ and $C_{i} = \partial_{B}(G_{i}, \mathcal{P}_{i})$, then $X = \partial_{B}(G, \bigcup_{i} \iota_{i}(\mathcal{P}_{i}))$. Then, in this special case, we have another description for the Bowditch boundary that was constructed by Dahmani on \cite{Da2}. In the special case where all vertex groups are hyperbolic and the spaces $C_{i}$ are their hyperbolic boundaries, then the space $X$ is the hyperbolic boundary of $G$. This case was proved in \cite{Sw} (Theorem 0.3 (2)).

\begin{cor}Let $G_{1},..,G_{n},G'_{1},...,G'_{n}$ be finitely generated groups, $G = G_{1}\ast...\ast G_{n}$, $G' = G'_{1}\ast...\ast G'_{n}$ and the inclusion maps $\iota_{i}: G_{i} \rightarrow G$ and $\iota'_{i}: G'_{i} \rightarrow G'$. If for every $i \in \{1,...,n\}$, $G_{i}$ is relatively hyperbolic with respect to $\mathcal{P}_{i}$, $G'_{i}$ is relatively hyperbolic with respect to $\mathcal{P}'_{i}$ and $\partial_{B}(G_{i},\mathcal{P}_{i})$ is homeomorphic to $\partial_{B}(G'_{i},\mathcal{P}'_{i})$, then $\partial_{B}(G,\bigcup_{i}\iota_{i}(\mathcal{P}_{i}))$ is homeomorphic to $\partial_{B}(G',\bigcup_{i}\iota'_{i}(\mathcal{P}'_{i}))$. \eod
\end{cor}

The Theorem 4.1 of \cite{MS} proves this corollary for the case of hyperbolic groups and hyperbolic boundaries.

\subsection{Space of ends}

Let $G$ be a group. We consider $G+_{\partial_{E}}Ends(G)$ the Freudenthal compactification of $G$ (i.e. the maximal equivariant perspective compactification with totally disconnected boundary).

\begin{prop}\label{ends}Let $G$ be a group and $p \in Ends(G)$ a bounded parabolic point of $G+_{\partial_{E}}Ends(G)$. Then $\#Ends(Stab \ p) = 1$.
\end{prop}

\begin{proof}Let $P$ be the set of bounded parabolic points, $P' \subseteq P$ be a representative set of orbits of $P$, $\mathcal{C} = \{Ends(Stab \ p)\}_{p\in P'}$, $\{Stab \ p+_{\partial_{p,E}}Ends(Stab \ p)\}_{p\in P'}$ the set of spaces with perspective property and obvious actions. Let $X = (G+_{\partial_{E}}Ends(G))\ltimes \mathcal{C}$. Since $X$ is a limit of totally disconnected spaces, it follows that it is totally disconnected as well. The blowup map is the unique continuous equivariant map $\alpha: X \rightarrow G+_{\partial_{E}}Ends(G)$, since $X$ is a compactification of $G$. Since $G+_{\partial_{E}}Ends(G)$ is the maximal perspective compactification of $G$ with totally disconnected boundary, it follows that $\alpha$ must be a homeomorphism. So, $\forall p \in P$, $\#\alpha^{-1}(p) = 1$. But $\alpha^{-1}(p)$ is homeomorphic to $Ends(Stab \ p)$. Thus, $\#Ends(Stab \ p) = 1$.
\end{proof}

\begin{cor}Let $\varphi: G \curvearrowright Z$ be an action by homeomorphisms on a Hausdorff compact space $Z$. Then there is a Hausdorff compact space $X$ together with a blowup map $\pi: X \rightarrow Z$ such that the inverse images of bounded parabolic points are totally disconnected and if there is another blowup map $\pi': X' \rightarrow Z$ with those properties, then there is a unique continuous equivariant map $\alpha: X \rightarrow X'$ that commutes with the blowup maps.
\end{cor}

\begin{proof}Let $P$ be the set of bounded parabolic points of $Z$, $P' \subseteq P$ be a representative set of orbits of $P$, $\mathcal{C} = \{Ends(Stab \ p)\}_{p\in P'}$ a set of compact spaces where the stabilizers act, $\{Stab \ p+_{\partial_{p,E}}Ends(Stab \ p)\}_{p\in P'}$ the set of spaces with perspective property and obvious actions. Let $X = Z\ltimes \mathcal{C}$. Let $p \in P$. Since $\pi^{-1}(p)$ is homeomorphic to $Ends (Stab \ p)$, then it is totally disconnected.

Let $\pi': X' \rightarrow Z$ be a blowup map with $\pi'^{-1}(p)$ totally disconnected for every $p \in P'$. Then, for $p \in P'$, there is a continuous map $id+\phi_{p}: Stab \ p+_{\partial_{p,E}}Ends(Stab \ p) \rightarrow Stab \ p +_{\delta'_{p}} \pi'(p)$, where $\{Stab \ p +_{\delta'_{p}} \pi'(p)\}_{p\in P'}$ is the family of perspective compactifications used on the construction of $\pi'$. The family $\phi = \{id+\phi_{p}\}_{p \in P'}$ induces the continuous map $\varphi \ltimes \phi: X \rightarrow X'$.

Let $\alpha: X \rightarrow X'$ be another equivariant continuous map that commutes with the blowups. Let $p \in P'$ and take the quotients $X_{p}$ and $X'_{p}$ that collapses all the fibers except $\pi^{-1}(p)$ and $\pi'^{-1}(p)$. This map $\alpha$ induces a map on the quotients $\alpha_{p}: X_{p} \rightarrow X'_{p}$. This map $\alpha_{p}$ is also induced by a map $id+\alpha_{p}|_{\pi^{-1}(p)}: Stab \ p+_{\partial_{p,E}}Ends(Stab \ p) \rightarrow Stab \ p +_{\delta'_{p}} \pi'(p)$. By the uniqueness of the map from the Freudenthal compactification, we have that $id+\alpha_{p}|_{\pi^{-1}(p)} = id+\phi_{p}$. Since $\alpha$ is induced by the maps $\{\alpha_{p}\}_{p\in P'}$, we have that $\alpha =  \varphi \ltimes \phi$.

Thus there is a unique continuous equivariant map that commutes with the blowups.
\end{proof}

We call such construction the ends blowup of $Z$. On the special case that  $G$ is finitely generated,  $\varphi$ has the convergence property and the attractor-sum compactification  $G+_{\partial_{c}}Z$ is geometric, then the action on $X$ has the convergence property.

\begin{obs}Let $q$ be a bounded parabolic point of $X$. We have, by \textbf{Proposition \ref{subindoparabolicos}}, that $\pi(q) = p \in P$. By the blowup construction $Stab \ q = Stab \ \pi_{p}(q)$, where $\pi_{p}: X \rightarrow X_{p}$ is the projection map. Since $\pi_{p}(q)$ is a bounded parabolic point of $Stab \ p +_{\partial_{p,E}}Ends(Stab \ p)$, the group $Stab \ \pi_{p}(q)$ is one-ended. Thus, all stabilizers of bounded parabolic points of $X$ are one-ended.
\end{obs}

We have a corollary that slightly generalizes two of Bowditch's results (Theorem 0.1 of \cite{Bo3} and Theorem 0.2 of \cite{Bo5}), when restricted to one-ended groups:

\begin{cor}\label{Bowditch}Let $G$ be a one-ended relatively hyperbolic group with respect to $\mathcal{P}$ such that every group in $\mathcal{P}$ is finitely presented and has no infinite torsion subgroup. Then we have:

\begin{enumerate}
    \item $\partial_{B}(G,\mathcal{P})$ is locally connected.
    \item Every global cut point of $\partial_{B}(G,\mathcal{P})$ is bounded parabolic.
\end{enumerate}
\end{cor}

\begin{obs}We show its proof just as an example of use of blowups, since there are stronger results that remove the hypothesis of $G$ being one-ended and the conditions in the subgroups in $\mathcal{P}$ (Theorems 1.2.1 and 1.2.2 of \cite{Das}).
\end{obs}

\begin{proof}Let $X$ be the ends blowup of $\partial_{B}(G,\mathcal{P})$. We have that $X$ is connected, since $G$ is one-ended. Let $p$ be a bounded parabolic point of $\partial_{B}(G,\mathcal{P})$. Since $Stab \ p$ is finitely presented, it is accessible (Theorem VI.6.3 of \cite{DD}), which implies that the action on $Ends(Stab \ p)$ is relatively hyperbolic \cite{Ge2}. So the action of $G$ on $X$ is relatively hyperbolic. Then the stabilizers of bounded parabolic points of $X$ are finitely presented (Proposition 9.10 of \cite{GP2}) and have no infinite torsion subgroups.

By Theorem 0.1 of \cite{Bo3} $X$ is locally connected. Since $\partial_{B}(G,\mathcal{P})$  is a quotient of a locally connected space and the quotient is a closed map, then $\partial_{B}(G,\mathcal{P})$ is also locally connected.

Let $x$ be a global cut point of $\partial_{B}(G,\mathcal{P})$ that is not bounded parabolic. Then $\#\pi^{-1}(p) = 1$, where $\pi: X \rightarrow \partial_{B}(G,\mathcal{P})$ is the blowup map. Then $\pi^{-1}(p)$ is also a cut point. But $\pi^{-1}(p)$ is a conical point, contradicting the Theorem 0.2 of \cite{Bo5}. Thus every global cut point of $\partial_{B}(G,\mathcal{P})$  is bounded parabolic. 
\end{proof}

\begin{prop}\label{blowupconect}Let $\pi: X \rightarrow Z$ be a blowup map and $P \subseteq Z$ the set of bounded parabolic points. If $Z$ is connected and $\forall p \in P$, $\pi^{-1}(p)$ is connected, then $X$ is connected.
\end{prop}

\begin{proof}Let $A$ be a clopen set of $X$ and $p \in P$ such that $A\cap \pi^{-1}(p) \neq \emptyset$. Since $\pi^{-1}(p)$ is connected, $\pi^{-1}(p) \subseteq A$. So $A$ is $\pi$-saturated. Then $\pi(A)$ is a clopen set of $Z$. Since $Z$ is connected, $\pi(A)$ is empty or equal to $Z$, which implies that $A$ is empty or equal to $X$. Thus $X$ is connected.
\end{proof}

The next two corollaries are also easy consequences of a Bowditch's Theorem (Theorem 10.1 of \cite{Bo4}).

\begin{cor}\label{endsofsubgroups}Let $G$ be a finitely generated group that is relatively hyperbolic with respect to a set of parabolic subgroups $\mathcal{P}$. If $\#Ends(G) > 1$ and $\partial_{B}(G,\mathcal{P})$ is connected, then there exists $H \in \mathcal{P}$ such that $\#Ends(H) > 1$.
\end{cor}

\begin{proof}Since the Freudenthal compactification $G+_{\partial_{E}}Ends(G)$ and the attractor-sum $G+_{\partial_{c}}\partial_{B}(G,\mathcal{P})$ are geometric \cite{Ge2}, there exists a geometric compactification $G+_{\partial}X$ and two continuous equivariant maps $\pi: X \rightarrow Ends(G)$, $\pi': X \rightarrow \partial_{B}(G,\mathcal{P})$. Since $\#Ends(G) > 1$, we have that $X$ is not connected. We have that the map $\pi'$ is a blowup map (\textbf{Theorem \ref{semidirectforconv}}) and $\partial_{B}(G,\mathcal{P})$ is connected. So, by the last proposition, there exists $H \in \mathcal{P}$ such that $\pi'^{-1}(\partial_{c}(H))$ is not connected. But $\pi'^{-1}(\partial_{c}(H)) = \partial(H)$. Thus $\#Ends(H) > 1$.
\end{proof}

\begin{cor}\label{withoutlocalcut}Let $G$ be a finitely generated group that is relatively hyperbolic with respect to a set of parabolic subgroups $\mathcal{P}$. If $\partial_{B}(G,\mathcal{P})$ is non trivial, connected and without local cut points, then $\#Ends(G) = 1$.
\end{cor}

\begin{proof}As above, let $G+_{\partial}X$ be a geometric compactification such that there exists two continuous equivariant applications $\pi: X \rightarrow Ends(G)$ and $\pi': X \rightarrow \partial_{B}(G,\mathcal{P})$. Let $P$ be the set of bounded parabolic points of $\partial_{B}(G,\mathcal{P})$. We have that $\pi'$ is a blowup map. By Corollary 0.2 of \cite{Bo3} the space $\partial_{B}(G,\mathcal{P})$ is locally connected. Since $\partial_{B}(G,\mathcal{P})$ is connected, locally connected and without local cut points, every group in $\mathcal{P}$ is one-ended (Proposition 3.3 of \cite{Da3}), which implies that $\forall p \in P$, $\pi'^{-1}(p)$ is connected. Then $X$ is connected, which implies that $Ends(G)$ is connected. Thus $\#Ends(G) = 1$.
\end{proof}

\begin{obs}Proposition 3.3 of \cite{Da3} says that if $\partial_{B}(G,\mathcal{P})$ is non trivial, connected, locally connected and without local cut points, then every group in $\mathcal{P}$ is one-ended. It is stated for the Merger curve but Dahmani's proof works in general.
\end{obs}

\subsection{Maximal perspective compactifications}

\begin{teo}\label{maximalpers}Let $G$ be a group and $G+_{\partial}Y$ its maximal perspective compactification. If $p \in Y$ is a bounded parabolic point and $Stab \ p+_{\partial_{p}} C_{p}$ is a compactification of $Stab \ p$ with the perspective property (i.e. $Stab \ p$ is dense on it), then $\#C_{p} = 1$.
\end{teo}

\begin{proof}Let $P$ be the set of bounded parabolic points, $P' \subseteq P$ be a representative set of orbits of $P$, $\mathcal{C} = \{C_{p}\}_{p\in P'}$ be a set of Hausdorff compact spaces, $\{Stab \ p+_{\partial_{p}} C_{p}\}_{p\in P'}$ a set of compact spaces with perspective property such that $\forall p \in P'$, $Stab \ p$ is dense. Let $X = (G+_{\partial}Y)\ltimes \mathcal{C}$. The blowup map is the unique continuous equivariant map $\alpha: X \rightarrow G+_{\partial}Y$, since $X$ is a compactification of $G$. Since $G+_{\partial}Y$ is the maximal perspective compactification of $G$, it follows that $\alpha$ must be an homeomorphism. So, $\forall p \in P$, $\#\alpha^{-1}(p) = 1$. But $\alpha^{-1}(p)$ is homeomorphic to $C_{p}$. Thus, $\#C_{p} = 1$.
\end{proof}

\subsection{CAT(0) spaces with isolated flats}

\begin{defi}A flat on a CAT(0) space $X$ is a subspace that is  isometric to the Euclidean space with dimension at least 2.

Let $G$ be a group acting properly discontinuously, cocompactly and by isometries on a proper CAT(0) space $X$. We say that $X$ has isolated flats with respect to a equivariant set of flats $\mathcal{F}$ if the following happens:

\begin{enumerate}
    \item There exists $D > 0$ such that every flat on $X$ is contained on the $D$-neighbourhood of an element in $\mathcal{F}$.
    \item For every $r > 0$, there exists $s > 0$ such that for every $F,F' \in \mathcal{F}$, $diam \ (\mathcal{B}(F,r)\cap \mathcal{B}(F',r)) < s$ (where $\mathcal{B}(A,\epsilon)$ is the $\epsilon$-neighbourhood of $A$). 
\end{enumerate}

\end{defi}

Let $\psi: G \curvearrowright X$ be a properly discontinuous cocompact action by isometries on a proper CAT(0) space $X$ with isolated flats with respect to $\mathcal{F}$. Let $\mathcal{P}$ be a set of representatives (up to conjugacy class) of its stabilizers. Hruska and Kleiner (Theorem 1.2.1 of \cite{HK}) proved  that $G$ is relatively hyperbolic with respect to $\mathcal{P}$. Tran (Theorem 1.1 of \cite{Tr}) showed that there exists a continuous equivariant map $\pi: \partial_{v}(X) \rightarrow \partial_{B}(G,\mathcal{P})$, where $\partial_{v}(X)$ is the visual boundary of $X$ . This map is injective on the inverse image of the set of conical points and for every bounded parabolic point $p \in \partial_{B}(G,\mathcal{P})$, there is a flat $F_{p} \in \mathcal{F}$ such that $\partial_{v}(F_{p}) = \pi^{-1}(p)$  and $Stab \ p = Stab \ F_{p}$. Let $p \in \partial_{B}(G,\mathcal{P})$ be a bounded parabolic point and $\psi_{p} = \psi|_{Stab \ p\times (\partial_{v}(X)-\partial_{v}(F_{p}))}$. Hruska and Ruane (Proposition 9.4 of \cite{HR}) proved that $\partial_{v}(X)\in EPers(\psi_{p})$. Then, by \textbf{Theorem \ref{geoblowup}}, we get:

\begin{prop}\label{catblowup}$\pi: \partial_{v}(X) \rightarrow \partial_{B}(G,\mathcal{P})$ is a blowup map \eod

\end{prop}

\begin{obs}Without the necessity of using \textbf{Theorem \ref{geoblowup}}, Hruska and Ruane (Proposition 8.12 of \cite{HR}) also proved that the map $\pi$ is topologically quasiconvex.

\end{obs}

\begin{cor}$\partial_{v}(X)$ is connected if and only if $\partial_{B}(G,\mathcal{P})$ is connected,

\end{cor}

\begin{proof}Immediate from \textbf{Proposition \ref{blowupconect}}.
\end{proof}

\end{document}